\documentclass[12pt]{amsart}
\usepackage[margin=1in]{geometry}
\usepackage{graphicx} 
\usepackage{amsthm}
\usepackage{amssymb}
\usepackage{amsmath}
\usepackage{tikz, pgfplots} 
\usetikzlibrary{calc}
\usetikzlibrary{shapes.geometric}
\usepackage{color}
\usepackage{soul}
\usepackage{hyperref}
\pgfplotsset{compat=1.18} 
\usepackage{lipsum}

\newtheorem{thm}{Theorem}[section]
\newtheorem{prop}[thm]{Proposition}
\newtheorem{cor}[thm]{Corollary}
\newtheorem{lemma}[thm]{Lemma}
\theoremstyle{definition}
\newtheorem{example}[thm]{Example}
\newtheorem{defn}[thm]{Definition}
\newtheorem{quest}[thm]{Question}

\definecolor{ao}{rgb}{0.0,0.5, 0.0}
\definecolor{amethyst}{rgb}{0.6, 0.4, 0.8}

\tikzset{%
auto, vertex/.style={circle,draw=black!100,thin,inner sep=0pt,minimum size=2mm, fill=black!0}, flood/.style={circle,draw=black!100,fill=cyan!100, thin,inner sep=0pt,minimum size=2mm}
}

\newcommand{\FG}{\mathcal{F}(G)}
\newcommand{\F}{F_G(x)}
\newcommand{\C}{\overline{C}}
\newcommand{\co}{\textrm{co}}
\newcommand{\cat}{\textrm{CEN}}
\newcommand{\tick}{\textrm{TICK}}

\usetikzlibrary{positioning}

\title{The Flood Polynomial of a Graph}

\author{Karin R. Saoub}
\address{\scriptsize{Department of Mathematics, Computer Science, and Physics, Roanoke College, Salem, VA}}
\email{saoub@roanoke.edu}
\author{Michael Weselcouch}
\address{\scriptsize{Department of Mathematics, Computer Science, and Physics, Roanoke College, Salem, VA}}
\email{weselcouch@roanoke.edu}
\author{Trey Wilhoit}
\address{\scriptsize{Department of Mathematics, Computer Science, and Physics, Roanoke College, Salem, VA}}
\email{ptwilhoit@mail.roanoke.edu}
\author{Jackson Wills}
\address{\scriptsize{Department of Mathematics, Computer Science, and Physics, Roanoke College, Salem, VA}}
\email{jtwills@mail.roanoke.edu}

\date{\today}

\begin{document}

\begin{abstract}
    The flood polynomial of a simple finite graph is a weight generating function that counts all flooding cascade sets of the graph.  The flood polynomial is inspired by the water mechanics in the video game Minecraft.  We give necessary conditions for two graphs to have the same flood polynomial.  We then provide a formula for the flood polynomials of certain families of graphs.  We will see that many flood polynomials can be expressed using a Fibonacci-like recurrence and in some cases are equal to Fibonacci or Lucas polynomials.  We then provide general examples of pairs of distinct graphs with the same flood polynomial.  In these examples, the flood polynomial will be expressed as the product of Fibonacci and Lucas polynomials.
\end{abstract}


\maketitle
%

\section{Introduction}

In this article, we introduce a graph polynomial, called the flood polynomial, which is based on the water mechanics in the video game Minecraft. The flood polynomial is a weight generating function that counts certain subsets of the vertices of a given graph.

In Minecraft, water blocks and air blocks have an interesting relationship. If an air block is neighbors with two or more water blocks, then the air block will convert to water, allowing a player to convert a large region of air to water with only a few initial water blocks.  Although regions in Minecraft must be contained in a grid, these flooding mechanics can be extended to general graphs.  


  In this paper, we determine families of graphs which have flood polynomials that are products of Fibonacci polynomials and Lucas polynomials, providing a new combinatorial interpretation of a well-known identity involving Fibonacci and Lucas polynomials. We give an explicit formula for finding the flood polynomials of these graphs in terms of these products.  This paper presents the first study of the flood polynomial, and we anticipate more results are possible beyond what is discussed in this article.

This paper is organized as follows: in Section 2, we provide the necessary background, including the definitions of cascade sets and the flood polynomial; in Section 3, we discuss properties of a graph that can be determined by its flood polynomial
; in Section 4, we give formulas for the  flood polynomials of certain families of graphs; 
in Section 5, we provide examples of pairs of distinct graphs with the same flood polynomial
; we conclude with presenting open questions in Section 6.

\section{Preliminaries}
We begin with some preliminaries about compositions and partitions, graphs, flood sets, and the flood polynomial. For more information, see \cite{Stanley}.

\subsection{Compositions and partitions}\label{comps and partitions}

A \textit{composition} $\alpha= (\alpha_1, \alpha_2, \dots, \alpha_k)$ of $n$ is a finite sequence of positive integers summing to $n$. The compositions of $n$ are in bijection with the subsets of $[n-1]$ in the following way: for any composition $\alpha$, define \[D(\alpha) = \{\alpha_1, \quad \alpha_1 + \alpha_2, \quad \dots, \quad \alpha_1 + \alpha_2 +\dots +\alpha_{k-1}\} \subseteq [n-1].\]

Likewise, for any subset $S = \{s_1, s_2, \dots, s_{k-1}\}\subseteq [n-1]$ with $s_1<s_2<\dots < s_{k-1}$, we can define the composition
\[\co(S) = (s_1, \quad s_2-s_1, \quad s_3-s_2, \quad \dots, \quad s_{k-1}-s_{k-2}, \quad n-s_{k-1}).\]

A \emph{partition} of $n$ is a composition of $n$ whose parts are in weakly decreasing order.  Given a composition $\alpha$ and partition $\lambda$, we write $\alpha \sim \lambda$ if $\lambda$ is formed by rearranging the parts of $\alpha$ in weakly decreasing order.  We use the notation $\alpha \vDash n$ if $\alpha$ is a composition of $n$ and $\lambda \vdash n$ if $\lambda$ is a partition of $n$.  We use $\ell(\alpha)$ to denote the number of parts of $\alpha$.

\subsection{Graphs}

A \emph{graph} $G$ consists of two sets: the vertex set, $V(G)$, and the edge set, $E(G)$.  An \emph{edge} is an unordered pair of vertices.  When it is clear what graph we are talking about, we will write $V$ for $V(G)$ and $E$ for $E(G)$.  Throughout the article we will use $n$ to represent the size of the graph, i.e. $n = |V|$.  We say that two vertices $a$ and $b$ are \emph{neighbors} in $G$ if $ab \in E(G)$, that is to say, $a$ and $b$ share an edge.  The \emph{degree} of a vertex $v$, denoted $\deg(v)$, is the number of neighbors of $v$. In this paper we only consider finite \emph{simple} graphs, which are graphs that do not contain any loops or multi-edges.  Consider the graphs shown below.  The graph on the left is simple, whereas the graph on the right is not simple since it has a multi-edge.

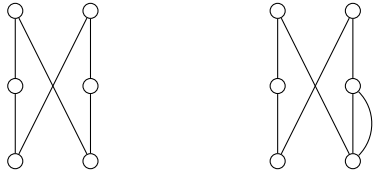
\begin{center}

\begin{tikzpicture}
\node (v1) at ( 0,0) [vertex] {};
\node (v2) at ( 1,0) [vertex] {};
\node (v4) at ( 0,1) [vertex] {};
\node (v5) at ( 1,1) [vertex] {};
\node (v6) at ( 0,2) [vertex] {};
\node (v7) at ( 1,2) [vertex] {};
\draw [-] (v1) to (v4);
\draw [-] (v4) to (v6);
\draw [-] (v2) to (v5);
\draw [-] (v5) to (v7);
\draw [-] (v1) to (v7);
\draw [-] (v2) to (v6);
\end{tikzpicture}
\hspace{2 cm}
\begin{tikzpicture}
\node (v1) at ( 0,0) [vertex] {};
\node (v2) at ( 1,0) [vertex] {};
\node (v4) at ( 0,1) [vertex] {};
\node (v5) at ( 1,1) [vertex] {};
\node (v6) at ( 0,2) [vertex] {};
\node (v7) at ( 1,2) [vertex] {};
\draw [-] (v1) to (v4);
\draw [-] (v4) to (v6);
\draw [-] (v2) to (v5);
\path (v2) edge [out=45,in=-45] (v5);
\draw [-] (v5) to (v7);
\draw [-] (v1) to (v7);
\draw [-] (v2) to (v6);
\end{tikzpicture}

\end{center}

\subsection{Cascade Sets and Flood Sets}

Given a graph $G$, a \emph{cascade set of $G$} is a subset of the vertices of $G$.  We are going to be interested in cascade sets that ``completely flood" the graph using the flooding mechanics of Minecraft.  In order to make this concept mathematically rigorous, we need the following definition.

\begin{defn}
For a cascade set $C$ and graph $G$, the \emph{cascade sequence} is a sequence of sets $C_0, C_1, \dots$ satisfying the following:
\begin{enumerate}
\item $C_0 = C$, and
\item for $k \geq 1$, \\ $C_k = C_{k-1} \cup \{x \in V \mid x \text{ has at least two neighbors in } C_{k-1} \}$.
\end{enumerate}
\end{defn}

In all future graphs, vertices in cascade sets will be denoted with an aqua coloring. 


\begin{example} \label{cascade set example}
Consider cascade set $C = \{v_1, v_4, v_6\}$ for the following graph. 

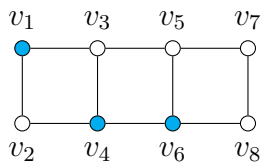
\begin{center}
\begin{tikzpicture}
    \node (v1) at (-4,0) [flood,  label=above:$v_1$]  {}; 
    \node (v2) at (-4,-1) [vertex, label=below: $v_2$] {};
    \node (v3) at (-3,0) [vertex, label=above: $v_3$] {};
    \node (v4) at (-3,-1) [flood, label=below: $v_4$] {};
    \node (v5) at (-2,0) [vertex, label=above: $v_5$] {};
    \node (v6) at (-2,-1) [flood, label=below: $v_6$] {};
    \node (v7) at (-1,0) [vertex, label=above: $v_7$] {};
    \node (v8) at (-1,-1) [vertex,label=below: $v_8$] {};

    \draw [-] (v1) to (v3);
    \draw [-] (v1) to (v2); 
    \draw [-] (v2) to (v4);
    \draw [-] (v3) to (v5);
    \draw [-] (v3) to (v4);
    \draw [-] (v4) to (v6);
    \draw [-] (v5) to (v6);
    \draw [-] (v6) to (v8);
    \draw [-] (v7) to (v8);
    \draw [-] (v5) to (v7);
\end{tikzpicture}
\end{center} 
We can see that both $v_2$ and $v_3$ have two neighbors in $C$, so $v_2$ and $v_3$ will flood. Therefore $C_1$ = $C$ $\cup$ $\{v_2, v_3\}$ as shown below. 
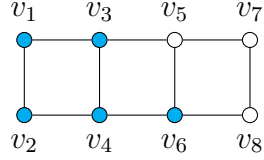
\begin{center}
\begin{tikzpicture}
    \node (v1) at (-4,0) [flood,  label=above:$v_1$]  {}; 
    \node (v2) at (-4,-1) [flood, label=below: $v_2$] {};
    \node (v3) at (-3,0) [flood, label=above: $v_3$] {};
    \node (v4) at (-3,-1) [flood, label=below: $v_4$] {};
    \node (v5) at (-2,0) [vertex, label=above: $v_5$] {};
    \node (v6) at (-2,-1) [flood, label=below: $v_6$] {};
    \node (v7) at (-1,0) [vertex, label=above: $v_7$] {};
    \node (v8) at (-1,-1) [vertex,label=below: $v_8$] {};

    \draw [-] (v1) to (v3);
    \draw [-] (v1) to (v2); 
    \draw [-] (v2) to (v4);
    \draw [-] (v3) to (v5);
    \draw [-] (v3) to (v4);
    \draw [-] (v4) to (v6);
    \draw [-] (v5) to (v6);
    \draw [-] (v6) to (v8);
    \draw [-] (v7) to (v8);
    \draw [-] (v5) to (v7);
\end{tikzpicture}
\end{center} 
Similarly, we can now see that $v_5$ has two neighbors in $C_1$, namely $v_3$ and $v_6$, so $v_5 \in C_2$.  You can check that no other vertices will flood in this step. Therefore $C_2$ = $C_1$ $\cup$ $\{v_5\}$.
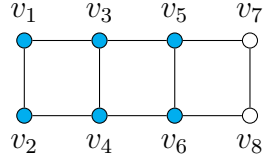
\begin{center}
\begin{tikzpicture}
    \node (v1) at (-4,0) [flood,  label=above:$v_1$]  {}; 
    \node (v2) at (-4,-1) [flood, label=below: $v_2$] {};
    \node (v3) at (-3,0) [flood, label=above: $v_3$] {};
    \node (v4) at (-3,-1) [flood, label=below: $v_4$] {};
    \node (v5) at (-2,0) [flood, label=above: $v_5$] {};
    \node (v6) at (-2,-1) [flood, label=below: $v_6$] {};
    \node (v7) at (-1,0) [vertex, label=above: $v_7$] {};
    \node (v8) at (-1,-1) [vertex,label=below: $v_8$] {};

    \draw [-] (v1) to (v3);
    \draw [-] (v1) to (v2); 
    \draw [-] (v2) to (v4);
    \draw [-] (v3) to (v5);
    \draw [-] (v3) to (v4);
    \draw [-] (v4) to (v6);
    \draw [-] (v5) to (v6);
    \draw [-] (v6) to (v8);
    \draw [-] (v7) to (v8);
    \draw [-] (v5) to (v7);
\end{tikzpicture}
\end{center} 
Since neither $v_7$ nor $v_8$ have two neighbors in $C_2$, no new vertices flood and $C_3$ = $C_2$.  In fact, in this example, for all $k \geq 3$, we have that $C_k = C_2$. 

\end{example}

From the definition of cascade sequence, it follows that if $C_k = C_{k+1}$ for some value $k$, then for all $j \geq k$, $C_k = C_j$, i.e., once two terms in the sequences are equal, all subsequent terms in the sequence are the same set. Similarly, since $G$ is finite and every cascade set is a subset of $V(G)$, there must exist a $k \in \mathbb{N}$ such that for all $j \geq k$, $C_j = C_k$. 
Let $\C$ denote the set to which  the cascade sequence starting with $C$ converges.  Note that if $C'$ is a term in the cascade sequence of $C$, then $\C = \overline{C'}$.  If $\C = V(G)$, then we say that $C$ \emph{completely floods} $G$ and that $C$ is a \emph{flooding cascade set}.  If $\C \neq V(G)$, then we say $C$ is a \emph{non-flooding cascade set}.  If $v \in V(G)-\C$, then $v$ is \emph{not flooded by $C$}.

In Example \ref{cascade set example}, we see that if $C = \{v_1, v_4, v_6\}$, then $\C = \{v_1,v_2,v_3, v_4, v_5, v_6\}$ and both the vertices $v_7$ and $v_8$ are not flooded by $C$.  Since there are some vertices that are not flooded by $C$, this means that $C$ is a non-flooding cascade set.

This leads us to the definition of the flood set.

\begin{defn}
The \emph{flood set} of a graph $G$, denoted $\FG$, is the set of cascade sets that completely flood $G$.
\end{defn}

\begin{example} \label{Flood Set}
The following seven cascade sets form the flood set of the corresponding graph. 
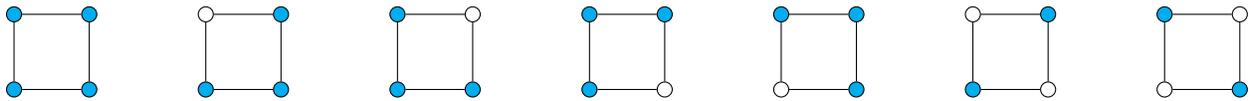
\begin{center}

\begin{tikzpicture}
    \node (v1) at (-4,0) [flood]  {}; 
    \node (v2) at (-4,-1) [flood] {};
    \node (v3) at (-3,0) [flood] {};
    \node (v4) at (-3,-1) [flood] {};

    \draw [-] (v1) to (v3);
    \draw [-] (v1) to (v2); 
    \draw [-] (v2) to (v4);
    \draw [-] (v3) to (v4);
\end{tikzpicture}
\hspace*{\fill}
\begin{tikzpicture}
    \node (v1) at (-8,0) [vertex]  {}; 
    \node (v2) at (-8,-1) [flood] {};
    \node (v3) at (-7,0) [flood] {};
    \node (v4) at (-7,-1) [flood] {};

    \draw [-] (v1) to (v3);
    \draw [-] (v1) to (v2); 
    \draw [-] (v2) to (v4);
    \draw [-] (v3) to (v4);
\end{tikzpicture}
\hspace*{\fill}
\begin{tikzpicture}
    \node (v1) at (-8,0) [flood]  {}; 
    \node (v2) at (-8,-1) [flood] {};
    \node (v3) at (-7,0) [vertex] {};
    \node (v4) at (-7,-1) [flood] {};

    \draw [-] (v1) to (v3);
    \draw [-] (v1) to (v2); 
    \draw [-] (v2) to (v4);
    \draw [-] (v3) to (v4);
\end{tikzpicture}
\hspace*{\fill}
\begin{tikzpicture}
    \node (v1) at (-8,0) [flood]  {}; 
    \node (v2) at (-8,-1) [flood] {};
    \node (v3) at (-7,0) [flood] {};
    \node (v4) at (-7,-1) [vertex] {};

    \draw [-] (v1) to (v3);
    \draw [-] (v1) to (v2); 
    \draw [-] (v2) to (v4);
    \draw [-] (v3) to (v4);
\end{tikzpicture}
\hspace*{\fill}
\begin{tikzpicture}
    \node (v1) at (-8,0) [flood]  {}; 
    \node (v2) at (-8,-1) [vertex] {};
    \node (v3) at (-7,0) [flood] {};
    \node (v4) at (-7,-1) [flood] {};

    \draw [-] (v1) to (v3);
    \draw [-] (v1) to (v2); 
    \draw [-] (v2) to (v4);
    \draw [-] (v3) to (v4);
\end{tikzpicture}
\hspace*{\fill}
\begin{tikzpicture}
    \node (v1) at (-8,0) [vertex]  {}; 
    \node (v2) at (-8,-1) [flood] {};
    \node (v3) at (-7,0) [flood] {};
    \node (v4) at (-7,-1) [vertex] {};

    \draw [-] (v1) to (v3);
    \draw [-] (v1) to (v2); 
    \draw [-] (v2) to (v4);
    \draw [-] (v3) to (v4);
\end{tikzpicture}
\hspace*{\fill}
\begin{tikzpicture}
    \node (v1) at (-8,0) [flood]  {}; 
    \node (v2) at (-8,-1) [vertex] {};
    \node (v3) at (-7,0) [vertex] {};
    \node (v4) at (-7,-1) [flood] {};

    \draw [-] (v1) to (v3);
    \draw [-] (v1) to (v2); 
    \draw [-] (v2) to (v4);
    \draw [-] (v3) to (v4);
\end{tikzpicture}
\end{center} 


\end{example}


Before defining the flood polynomial, which is the remaining focus of the paper, we prove some basic results about cascade sets.

\begin{prop}
    If $C$ is a cascade set of $G$ and $|C| = 1$, then $C \in \FG$ if and only if $G$ has a single vertex.
\end{prop}

\begin{proof}
     Suppose that $C$ is a one-element cascade set and $|G| > 1$. 
     Therefore, $V - C$ is non-empty and contains no element that has two neighbors in $C$. This means $C = C_1$ in the cascade sequence, hence $C = \overline{C} \neq V$. So $C \notin \FG$. 
    
    Now suppose $|G| = 1$ and $C$ is a one-element cascade set. Since $C$ is a subset of $V$ and $C$ and $V$ have the same number of elements, they must be equal. Hence $C \in \FG$, as desired. 
\end{proof}

\begin{prop} \label{flooding superset}
     If $C \in \FG$ and $C \subseteq C'$, then $C' \in \FG$.
\end{prop}

\begin{proof}
    For contradiction, let $C$ be the set with the most elements such that $C \in \FG$ and there exists $C' \supseteq C$ with $C' \notin \FG$.  Note that for this to be possible, $|C| < n$.

    Since $C \in \FG$ and $C \neq V$, it follows that $|C_1| > |C|$. Since $C$ was picked to be the set with the most elements with the desired property, any superset of $C_1$ necessarily floods $G$.  We will now show that $C_1'$ is a superset of $C_1$. 

    Let $v \in C_1$.  If $v \in C$, then $v \in C'$.  If $v \notin C$, then $v$ has two neighbors that are in $C$.  Since $C' \supseteq C$, this means that $v$ has two neighbors that are in $C'$.  Therefore $v \in C_1'$ and $C_1' \in \FG$ and hence $C' \in \FG$.  This is a contradiction.  Therefore if $C \in \FG$ and $C \subseteq C'$, then $C' \in \FG$ as desired. 
\end{proof}

We say that a flooding cascade set $C \in \FG$ is \emph{minimal} if for all $K \in \FG$, $K \subseteq C$ implies $K = C$.  That is to say, if $C$ is a minimal flooding cascade set, then no proper subset of $C$ is a flooding cascade subset.  Note that a graph can have minimal flooding cascade sets of different sizes. For example, consider the path graph with five vertices.  This graph has two minimal flooding cascade sets, one with three elements and one with four elements as we see in the following example.

\begin{example}
The two minimal flooding sets for the path graph with five vertices are shown below.  Notice that they are sets of different sizes.

\begin{center}
\hspace*{\fill}
\begin{tikzpicture}
    \node (v1) at (0,0) [flood]  {}; 
    \node (v2) at (1,0) [vertex] {};
    \node (v3) at (2,0) [flood] {};
    \node (v4) at (3,0) [vertex] {};
    \node (v5) at (4,0) [flood] {};

    \draw [-] (v1) to (v2);
    \draw [-] (v2) to (v3); 
    \draw [-] (v3) to (v4);
    \draw [-] (v4) to (v5);
\end{tikzpicture}
\hspace*{\fill}
\begin{tikzpicture}
    \node (v1) at (0,0) [flood]  {}; 
    \node (v2) at (1,0) [flood] {};
    \node (v3) at (2,0) [vertex] {};
    \node (v4) at (3,0) [flood] {};
    \node (v5) at (4,0) [flood] {};

    \draw [-] (v1) to (v2);
    \draw [-] (v2) to (v3); 
    \draw [-] (v3) to (v4);
    \draw [-] (v4) to (v5);
\end{tikzpicture}
\hspace*{\fill}
\end{center}
We can see that no proper subset of either of these cascade sets will flood the graph. 
\end{example}


There is no relation between the number of
vertices in a graph and the size of its minimal flooding sets. 
We will see in the following example, for all $n \geq 2$, there are graphs that have two-element flooding cascade sets.  Let $K_n$ denote the complete graph with $n$ vertices and let $C$ be any two vertices of $K_n$.
Since every vertex that's not in $C$ is neighbors with both elements of $C$, $C_1$ contains every vertex.

\begin{example} Consider $K_{6}$ shown below.

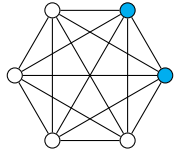
\begin{center}
\begin{tikzpicture}
    \node (v1) at (1,0) [flood]  {}; 
    \node (v2) at (.5,.866) [flood] {};
    \node (v3) at (-.5,.866) [vertex] {};
    \node (v4) at (-1,0) [vertex] {};
    \node (v5) at (-.5,-.866) [vertex] {};
    \node (v6) at (.5,-.866) [vertex] {};

    \draw [-] (v1) to (v2);
    \draw [-] (v2) to (v3); 
    \draw [-] (v3) to (v4);
    \draw [-] (v4) to (v5);
    \draw [-] (v5) to (v6);
    \draw [-] (v6) to (v1);
    \draw [-] (v1) to (v3); 
    \draw [-] (v1) to (v4);
    \draw [-] (v1) to (v5);
    \draw [-] (v2) to (v4);
    \draw [-] (v2) to (v5);
    \draw [-] (v2) to (v6);
    \draw [-] (v3) to (v5);
    \draw [-] (v3) to (v6);
    \draw [-] (v6) to (v4);
\end{tikzpicture}
\end{center}
We can see that every unflooded vertex is neighbors with both elements of the cascade set, so all unflooded vertices in $K_6$ will immediately  flood.

\end{example}

The \emph{diameter} of a graph is the 
the length of the shortest path between the most distanced vertices.
In the case of the complete graph, the diameter is $1$ since every pair of vertices share an edge.  The example below illustrates that a graph with a two-element flooding cascade set can have an arbitrarily large diameter.

\begin{example} \label{diameter}
The graph $T_n$ (see Section \ref{triangle graph} for a discussion of this family of graphs) has a diameter of $\lfloor \frac{n}{2} \rfloor$, 
but has a two-element flooding cascade set, $C = \{v_1, v_2\}$. 
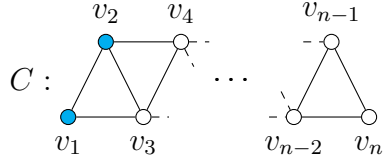
\begin{center}
\begin{tikzpicture}
\node (v1) at (-2, -1) [flood, label=below: $v_1$] {}; 
\node (v2) at (-1.5, 0) [flood, label=above: $v_2$] {};
\node (v3) at (-1, -1) [vertex, label=below: $v_3$] {};
\node (v4) at (-.5, 0) [vertex, label=above: $v_4$] {};
\node (v5) at (1, -1) [vertex, label=below: $v_{n-2}$] {};
\node (v6) at (1.5, 0) [vertex, label=above: $v_{n-1}$] {};
\node (v7) at (2, -1) [vertex, label=below: $v_n$] {}; 

\node at (-2.5, -.5) {$C:$}; 
\node at (.2, -.5) {$\dots$}; 
\node at (.2, -.5) {$\dots$}; 

\draw[-] (v1) to (v2); 
\draw[-] (v1) to (v3); 
\draw[-] (v2) to (v3); 
\draw[-] (v2) to (v4); 
\draw[-] (v3) to (v4); 
\draw[dashed, -] (v4) to (-.33, -.33);
\draw[dashed, -] (.83, -.66) to (v5); 
\draw[dashed, -] (v4) to (-.125, 0);
\draw[dashed, -] (1.125, 0) to (v6) ;
\draw[dashed, -] (v3) to (-.66, -1); 
\draw[dashed, -] (.67, -1) to (v5); 
\draw[-] (v5) to (v6); 
\draw[-] (v6) to (v7); 
\draw[-] (v5) to (v7); 
\end{tikzpicture}
\end{center}
We can see that $v_3$ is neighbors with both $v_1$ and $v_2$, so $v_3 \in C_1$.  Similarly, $v_4$ is an element of  $C_2$, and eventually, $v_n \in C_{n-2}.$

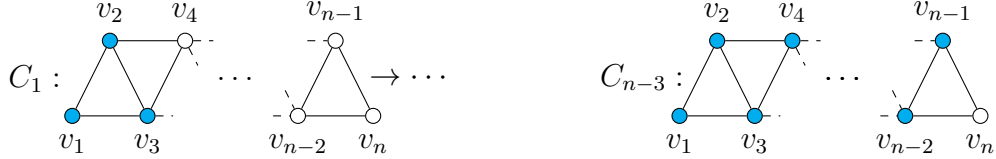
\begin{center}
\hspace*{\fill}
\begin{tikzpicture}
\node (v1) at (-2, -1) [flood, label = below: $v_1$] {};
\node (v2) at (-1.5, 0) [flood, label=above: $v_2$] {};
\node (v3) at (-1, -1) [flood, label=below: $v_3$] {};
\node (v4) at (-.5, 0) [vertex, label=above: $v_4$] {};
\node (v5) at (1, -1) [vertex, label=below: $v_{n-2}$] {};
\node (v6) at (1.5, 0) [vertex, label=above: $v_{n-1}$] {};
\node (v7) at (2, -1) [vertex, label=below: $v_n$] {}; 

\node at (-2.5, -.5) {$C_1:$}; 
\node at (2.5, -.5) {$\rightarrow \cdots$};

\node at (.2, -.5) {$\dots$}; 

\draw[-] (v1) to (v2); 
\draw[-] (v1) to (v3); 
\draw[-] (v2) to (v3); 
\draw[-] (v2) to (v4); 
\draw[-] (v3) to (v4); 
\draw[dashed, -] (v4) to (-.33, -.33);
\draw[dashed, -] (.83, -.66) to (v5); 
\draw[dashed, -] (v4) to (-.125, 0);
\draw[dashed, -] (1.125, 0) to (v6) ;
\draw[dashed, -] (v3) to (-.66, -1); 
\draw[dashed, -] (.67, -1) to (v5); 
\draw[-] (v5) to (v6); 
\draw[-] (v6) to (v7); 
\draw[-] (v5) to (v7); 
\end{tikzpicture}
\hspace*{\fill}
\begin{tikzpicture}
\node (v1) at (-2, -1) [flood, label=below: $v_1$] {}; 
\node (v2) at (-1.5, 0) [flood, label=above: $v_2$] {};
\node (v3) at (-1, -1) [flood, label=below: $v_3$] {};
\node (v4) at (-.5, 0) [flood, label=above: $v_4$] {};
\node (v5) at (1, -1) [flood, label=below: $v_{n-2}$] {};
\node (v6) at (1.5, 0) [flood, label=above: $v_{n-1}$] {};
\node (v7) at (2, -1) [vertex, label=below: $v_n$] {}; 

\node at (-2.5, -.5) {$C_{n-3}:$};

\node at (.2, -.5) {$\dots$}; 
\node at (.2, -.5) {$\dots$}; 

\draw[-] (v1) to (v2); 
\draw[-] (v1) to (v3); 
\draw[-] (v2) to (v3); 
\draw[-] (v2) to (v4); 
\draw[-] (v3) to (v4); 
\draw[dashed, -] (v4) to (-.33, -.33);
\draw[dashed, -] (.83, -.66) to (v5); 
\draw[dashed, -] (v4) to (-.125, 0);
\draw[dashed, -] (1.125, 0) to (v6) ;
\draw[dashed, -] (v3) to (-.66, -1); 
\draw[dashed, -] (.67, -1) to (v5); 
\draw[-] (v5) to (v6); 
\draw[-] (v6) to (v7); 
\draw[-] (v5) to (v7); 
\end{tikzpicture}
\hspace*{\fill}
\end{center}

We give a general classification of the flood set of $T_n$ in Lemma \ref{distance of 4 condition}.
\end{example}

\subsection{Flood Polynomial}

Now that we have established some basic properties of flooding cascade sets, we introduce our main area of study.

\begin{defn}\label{Def Flood Polynomial}
   The \emph{flood polynomial} of a graph $G$, denoted by $\F$, is defined by 
   \[ \F = \sum_{C \in \FG}x^{|C|}.\]
\end{defn}

Since all flooding cascade sets are subsets of the vertices of $G$, it follows that for all $k > n$, the coefficient of $x^k$ in $\F$ is 0. Therefore $\F$ is indeed a polynomial and its degree is at most $n$.  In fact we will see that the degree of $\F$ is equal to $n$ (see Proposition \ref{size of graph}).  We can write $\F$ as $\displaystyle \F = \sum_{k=0}^n c_kx^k$, where $c_k$ is the number of $k$-element flooding cascade sets of $G$.  The following result follows directly from the fact that flooding cascade sets are subsets of the vertices of $G$.

\begin{prop} \label{coefficients}
    If $G$ is a graph with $n$ vertices and \\$\displaystyle \F = \sum_{k=0}^n c_kx^k$, then $0 \leq c_k \leq \binom{n}{k}$.
\end{prop}


\begin{example}
 Let $G$ be the following graph.
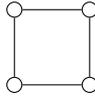
\begin{center}
 \begin{tikzpicture}
    \node (v1) at (0,0) [vertex] {}; 
    \node (v2) at (1,0) [vertex] {};
    \node (v3) at (1,-1) [vertex] {};
    \node (v4) at (0,-1) [vertex] {};

    \draw [-] (v1) to (v2);
    \draw [-] (v2) to (v3); 
    \draw [-] (v3) to (v4);
    \draw [-] (v4) to (v1);
\end{tikzpicture}
\end{center}


It follows from the flood set shown in Example \ref{Flood Set}, that $\F = x^4 + 4x^3 + 2x^2.$
 
\end{example}

\begin{prop} \label{disconnected}
    If $G$ is the disjoint union of graphs $H$ and $K$, i.e. $G = H \oplus K$, then $\F = F_H(x)\cdot F_K(x)$.
\end{prop}

\begin{proof}
Suppose $C$ is a cascade set of $G$ and let $C|_H$ be the elements of $C$ that are in $H$.  Similarly, let $C|_K$ be the elements of $C$ that are in $K$.  Since there are no edges between vertices in $H$ and vertices in $K$, then $C \in \FG$ if and only if $C|_H \in \mathcal{F}(H)$ and $C|_K\in \mathcal{F}(K)$.
%
\end{proof}

\section{Properties determined by the flood polynomial}

In this section, we discuss graph properties that can be determined by its flood polynomial.  We state necessary conditions for two graphs to have the same flood polynomials and some properties that cannot be determined by the flood polynomial.  The first observation follows directly from the fact that the vertex set of a graph is a maximum size flooding cascade set.

\begin{prop} \label{numvertices}
    The number of vertices of $G$ is determined by $\F$.
\end{prop}

The next property follows immediately from Definition \ref{Def Flood Polynomial}.

\begin{prop} \label{size of graph}
    The size of $\FG$ determined by $\F$.
\end{prop}

\begin{proof}
    It follows from the definition of $\F$ that $|\FG| = F_G(1)$.
\end{proof}

While it may be reasonable to expect that the number of minimal flooding cascade sets of $G$ may be determined by $\FG$, the following example demonstrates this not to be the case.  Additionally, the number of elements in each minimal flooding cascade sets cannot be determined by $\FG$.

\begin{example} \label{minimal flooding cascade sets example}
We will see in Section \ref{Path graphs with an even number of vertices} that the following graphs both have 
\\ $x^8+6x^7+10x^6+4x^5$
as their flood polynomial,
but not the same number of minimal flooding cascade sets.  

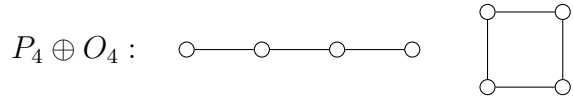
\vspace{.5 cm}
\begin{center}
\begin{tikzpicture}
\node at (-1, 0) {$P_8:$};
\node (v1) at ( 0,0) [vertex] {};
\node (v2) at ( 1,0) [vertex] {};
\node (v3) at ( 2,0) [vertex] {};
\node (v4) at ( 3,0) [vertex] {};
\node (v5) at ( 4,0) [vertex] {};
\node (v6) at ( 5,0) [vertex] {};
\node (v7) at ( 6,0) [vertex] {};
\node (v8) at ( 7,0) [vertex] {};
\draw [-] (v1) to (v2);
\draw [-] (v2) to (v3);
\draw [-] (v3) to (v4);
\draw [-] (v4) to (v5);
\draw [-] (v5) to (v6);
\draw [-] (v6) to (v7);
\draw [-] (v7) to (v8);
\end{tikzpicture}
\end{center}
\vspace{.5 cm}
\begin{center}
\begin{tikzpicture}
 \node at (-1.5, .5) {$P_4 \oplus O_4:$};
\node (v1) at ( 0,.5) [vertex] {};
\node (v2) at ( 1,.5) [vertex] {};
\node (v3) at ( 2,.5) [vertex] {};
\node (v4) at ( 3,.5) [vertex] {};
\node (v5) at ( 4,0) [vertex] {};
\node (v6) at ( 5,0) [vertex] {};
\node (v7) at ( 4,1) [vertex] {};
\node (v8) at ( 5,1) [vertex] {};
\draw [-] (v1) to (v2);
\draw [-] (v2) to (v3);
\draw [-] (v3) to (v4);
\draw [-] (v5) to (v7);
\draw [-] (v5) to (v6);
\draw [-] (v6) to (v8);
\draw [-] (v7) to (v8);
\end{tikzpicture}
\end{center}

The five minimal flooding cascade sets of $P_8$ are:

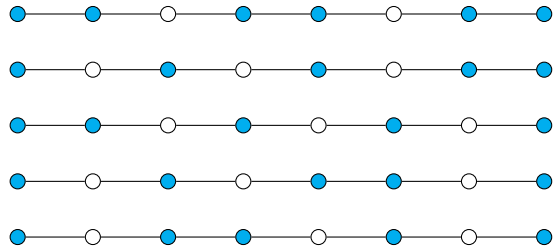
\begin{center}
\begin{tikzpicture}
\node (v1) at ( 0,0) [flood] {};
\node (v2) at ( 1,0) [flood] {};
\node (v3) at ( 2,0) [vertex] {};
\node (v4) at ( 3,0) [flood] {};
\node (v5) at ( 4,0) [flood] {};
\node (v6) at ( 5,0) [vertex] {};
\node (v7) at ( 6,0) [flood] {};
\node (v8) at ( 7,0) [flood] {};
\draw [-] (v1) to (v2);
\draw [-] (v2) to (v3);
\draw [-] (v3) to (v4);
\draw [-] (v4) to (v5);
\draw [-] (v5) to (v6);
\draw [-] (v6) to (v7);
\draw [-] (v7) to (v8);
\end{tikzpicture}

\vspace{.25 cm}

\begin{tikzpicture}
\node (v1) at ( 0,0) [flood] {};
\node (v2) at ( 1,0) [vertex] {};
\node (v3) at ( 2,0) [flood] {};
\node (v4) at ( 3,0) [vertex] {};
\node (v5) at ( 4,0) [flood] {};
\node (v6) at ( 5,0) [vertex] {};
\node (v7) at ( 6,0) [flood] {};
\node (v8) at ( 7,0) [flood] {};
\draw [-] (v1) to (v2);
\draw [-] (v2) to (v3);
\draw [-] (v3) to (v4);
\draw [-] (v4) to (v5);
\draw [-] (v5) to (v6);
\draw [-] (v6) to (v7);
\draw [-] (v7) to (v8);
\end{tikzpicture}

\vspace{.25 cm}

\begin{tikzpicture}
\node (v1) at ( 0,0) [flood] {};
\node (v2) at ( 1,0) [flood] {};
\node (v3) at ( 2,0) [vertex] {};
\node (v4) at ( 3,0) [flood] {};
\node (v5) at ( 4,0) [vertex] {};
\node (v6) at ( 5,0) [flood] {};
\node (v7) at ( 6,0) [vertex] {};
\node (v8) at ( 7,0) [flood] {};
\draw [-] (v1) to (v2);
\draw [-] (v2) to (v3);
\draw [-] (v3) to (v4);
\draw [-] (v4) to (v5);
\draw [-] (v5) to (v6);
\draw [-] (v6) to (v7);
\draw [-] (v7) to (v8);
\end{tikzpicture}

\vspace{.25 cm}

\begin{tikzpicture}
\node (v1) at ( 0,0) [flood] {};
\node (v2) at ( 1,0) [vertex] {};
\node (v3) at ( 2,0) [flood] {};
\node (v4) at ( 3,0) [vertex] {};
\node (v5) at ( 4,0) [flood] {};
\node (v6) at ( 5,0) [flood] {};
\node (v7) at ( 6,0) [vertex] {};
\node (v8) at ( 7,0) [flood] {};
\draw [-] (v1) to (v2);
\draw [-] (v2) to (v3);
\draw [-] (v3) to (v4);
\draw [-] (v4) to (v5);
\draw [-] (v5) to (v6);
\draw [-] (v6) to (v7);
\draw [-] (v7) to (v8);
\end{tikzpicture}

\vspace{.25 cm}

\begin{tikzpicture}
\node (v1) at ( 0,0) [flood] {};
\node (v2) at ( 1,0) [vertex] {};
\node (v3) at ( 2,0) [flood] {};
\node (v4) at ( 3,0) [flood] {};
\node (v5) at ( 4,0) [vertex] {};
\node (v6) at ( 5,0) [flood] {};
\node (v7) at ( 6,0) [vertex] {};
\node (v8) at ( 7,0) [flood] {};
\draw [-] (v1) to (v2);
\draw [-] (v2) to (v3);
\draw [-] (v3) to (v4);
\draw [-] (v4) to (v5);
\draw [-] (v5) to (v6);
\draw [-] (v6) to (v7);
\draw [-] (v7) to (v8);
\end{tikzpicture}

\end{center}

The four minimal flooding cascade sets of $P_4 \oplus O_4$ are:

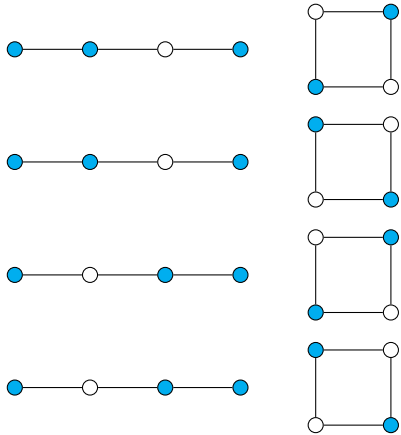
\begin{center}
\begin{tikzpicture}
\node (v1) at ( 0,.5) [flood] {};
\node (v2) at ( 1,.5) [flood] {};
\node (v3) at ( 2,.5) [vertex] {};
\node (v4) at ( 3,.5) [flood] {};
\node (v5) at ( 4,0) [flood] {};
\node (v6) at ( 5,0) [vertex] {};
\node (v7) at ( 4,1) [vertex] {};
\node (v8) at ( 5,1) [flood] {};
\draw [-] (v1) to (v2);
\draw [-] (v2) to (v3);
\draw [-] (v3) to (v4);
\draw [-] (v5) to (v7);
\draw [-] (v5) to (v6);
\draw [-] (v6) to (v8);
\draw [-] (v7) to (v8);
\end{tikzpicture}

\vspace{.25 cm}

\begin{tikzpicture}
\node (v1) at ( 0,.5) [flood] {};
\node (v2) at ( 1,.5) [flood] {};
\node (v3) at ( 2,.5) [vertex] {};
\node (v4) at ( 3,.5) [flood] {};
\node (v5) at ( 4,0) [vertex] {};
\node (v6) at ( 5,0) [flood] {};
\node (v7) at ( 4,1) [flood] {};
\node (v8) at ( 5,1) [vertex] {};
\draw [-] (v1) to (v2);
\draw [-] (v2) to (v3);
\draw [-] (v3) to (v4);
\draw [-] (v5) to (v7);
\draw [-] (v5) to (v6);
\draw [-] (v6) to (v8);
\draw [-] (v7) to (v8);
\end{tikzpicture}

\vspace{.25 cm}
\begin{tikzpicture}
\node (v1) at ( 0,.5) [flood] {};
\node (v2) at ( 1,.5) [vertex] {};
\node (v3) at ( 2,.5) [flood] {};
\node (v4) at ( 3,.5) [flood] {};
\node (v5) at ( 4,0) [flood] {};
\node (v6) at ( 5,0) [vertex] {};
\node (v7) at ( 4,1) [vertex] {};
\node (v8) at ( 5,1) [flood] {};
\draw [-] (v1) to (v2);
\draw [-] (v2) to (v3);
\draw [-] (v3) to (v4);
\draw [-] (v5) to (v7);
\draw [-] (v5) to (v6);
\draw [-] (v6) to (v8);
\draw [-] (v7) to (v8);
\end{tikzpicture}

\vspace{.25 cm}
\begin{tikzpicture}
\node (v1) at ( 0,.5) [flood] {};
\node (v2) at ( 1,.5) [vertex] {};
\node (v3) at ( 2,.5) [flood] {};
\node (v4) at ( 3,.5) [flood] {};
\node (v5) at ( 4,0) [vertex] {};
\node (v6) at ( 5,0) [flood] {};
\node (v7) at ( 4,1) [flood] {};
\node (v8) at ( 5,1) [vertex] {};
\draw [-] (v1) to (v2);
\draw [-] (v2) to (v3);
\draw [-] (v3) to (v4);
\draw [-] (v5) to (v7);
\draw [-] (v5) to (v6);
\draw [-] (v6) to (v8);
\draw [-] (v7) to (v8);
\end{tikzpicture}

\end{center}

\end{example}

While the flood polynomial of a disconnected graph is the product of the flood polynomials of each component of the graph (see Proposition \ref{disconnected}), the previous example illustrates that the number of components of a graph is not determined by the flood polynomial. In fact, the following example gives the smallest case in which two graphs have the same flood polynomial, but different numbers of components.

\begin{example} \label{2 element graphs}
The following graphs have a flood polynomial equal to $x^2$, but different numbers of components.

\begin{center}
\begin{tikzpicture}
 [auto, vertex/.style={circle,draw=black!100,thin,inner sep=0pt,minimum size=1.5mm}, flood/.style={circle,draw=black!100,fill=cyan!100, thin,inner sep=0pt,minimum size=1.5mm}]
\node at (-1, 0) {$P_2:$};
\node (v1) at ( 0,0) [vertex] {};
\node (v2) at ( 1,0) [vertex] {};
\draw [-] (v1) to (v2);

\node at (3, 0) {$P_1 \oplus P_1:$};
\node (v1) at (4.5,0) [vertex] {};
\node (v2) at ( 5.5,0) [vertex] {};

\end{tikzpicture}
\end{center}


\end{example}

In Section \ref{Graphs with the Same Flood Polynomial} we will give families of connected graphs who share flood polynomials with disconnected graphs.

\subsection{Isolated points, leaves, and triggers} \label{leaves and triggers}

A vertex is called an \emph{isolated point} if it has no neighbors.  It is called a \emph{leaf} if it has exactly one neighbor.  The total number of isolated points and leaves in the each of the graphs in Example \ref{2 element graphs} is two. The result below demonstrates that graphs with the same flood polynomials also have the same number of leaves and isolated points.

\begin{thm} \label{numleaves}
 If $\displaystyle \F = \sum_{k=0}^n c_kx^k$, then the total number of isolated points and leaves is $n - c_{n-1}$.
\end{thm}

\begin{proof}

Let $L$ be the total number of isolated points and leaves in $G$.  We want to show that $L=n - c_{n-1}$.  Note that $n - c_{n-1}$ is the number of $(n-1)$-element non-flooding cascade sets of $G$.  This is because the total number of $(n-1)$-element cascades sets (flooding or non-flooding) of $G$ is $\binom{n}{n-1} =n$ and $c_{n-1}$ is the number of $(n-1)$-element flooding cascade sets.  Therefore, if we show that an $(n-1)$-element cascade set is flooding if and only if it contains all of the isolated points and leaves of $G$, then we have proven the result since that would imply that the number of $(n-1)$-element non-flooding cascade sets is equal to $\binom{L}{1}=L$.

Suppose $C$ is a cascade set with $n-1$ elements and suppose $v$ is the vertex that is not in $C$.  If $v$ is an isolated point or a leaf, then it does not have two neighbors in $C$, so $v \notin C_1$ and $C = C_1$.  Therefore $\C \neq V$ and $C \notin \FG$. There are $L$ different possibilities for $v$ in this case.   If $v$ is not an isolated point or leaf, then it does have two neighbors in $C$.  Therefore $v \in C_1$ and $C_1 = V$.  Hence $C \in \FG$.

Therefore the number of $(n-1)$-element non-flooding cascade sets is equal to the number of vertices that are isolated points or leaves.
\end{proof}

\begin{cor} \label{leaves}
    The total number of isolated points and leaves of $G$ is determined by $\F$.
\end{cor}

The following result follows from the proof of Theorem \ref{numleaves}.

\begin{cor} \label{leaves in cascade set}
    If $C \in \FG$, then $C$ contains all of $G$'s isolated points and leaves.
\end{cor}

We now turn our attention to certain two-element subgraphs, which provide a natural way to extend our results about leaves and isolated points.

\begin{defn} \label{deftrigger}A \emph{trigger} is a 
set of two vertices $\{v_i, v_j\}$, where $v_i$ and $v_j$ share an edge and both have degree 2.
\end{defn}
Before proving that the flood polynomial determines the number of triggers, we consider the following example to solidify the concept of a trigger. 

\begin{example} The following graph contains two triggers, $\{v_{1}, v_{4}\}$ and  $\{v_{3}, v_{4}\}$.
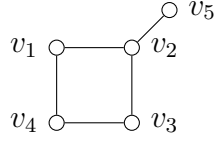
    \begin{center}
 \begin{tikzpicture}
    \node (v1) at (0,0) [vertex,label=left: $v_1$] {}; 
    \node (v2) at (1,0) [vertex, label=right: $v_2$] {};
    \node (v3) at (1,-1) [vertex, label=right: $v_3$] {};
    \node (v4) at (0,-1) [vertex, label=left: $v_4$] {};
    \node (v5) at (1.5, 0.5) [vertex, label=right: $v_5$] {};

    \draw [-] (v1) to (v2);
    \draw [-] (v2) to (v3); 
    \draw [-] (v3) to (v4);
    \draw [-] (v4) to (v1);
    \draw [-] (v2) to (v5);
\end{tikzpicture}
\end{center}
Even though $v_1$ and $v_3$ both have degree two, $\{v_1,v_3\}$ is not a trigger since $v_1$ and $v_3$ are not neighbors.
\end{example}

Triggers play an important role in determining whether a cascade set completely floods the graph, as shown below.

\begin{lemma} \label{trigger flood}
    Let $u$ and $v$ be vertices of degree at least $2$ and let $C = V-\{u, v\}$. Then $C \in \FG$ if and only if $\{u, v\}$ is not a trigger.
\end{lemma}
%

\begin{proof}
    Let $C$ be a cascade set of $G$.  If $\{u, v\}$ is a trigger, then $u$ and $v$ are neighbors and they each have one other neighbor in the graph.  This means $u \notin C_1$ and $v \notin C_1$ since neither one has two neighbors in $C$.  Therefore $C = C_1$ and $C \notin \FG$.

    If $\{u, v\}$ is not a trigger, then either $u$ and $v$ are not neighbors, or $u$ and $v$ are neighbors with at least one with degree greater than 2 (our assumption is that $u$ and $v$ are not isolated points or leaves).

    Suppose that $u$ and $v$ are not neighbors.  Since they both have degree at least $2$, it follows that they both have $2$ neighbors in $C$.  Therefore $C_1 = V$ and $C \in \FG$.

    Now suppose without loss of generality that that $u$ and $v$ are neighbors but the degree of $u$ is greater than 2.  This means that $u$ has at least $2$ neighbors in $C$ so $u \in C_1$.  Since $u \in C_1$ and $v$ has at least one neighbor in $C$, it follows that $v \in C_2$ and $C_2 = V$.  Therefore $C \in \FG$.
\end{proof}


Combining this result with the contrapositive of Proposition \ref{flooding superset} gives the following.

\begin{cor} \label{Jackson Cor}
    If $C$ is a cascade set and $V-C$ contains a trigger, then $C \notin \FG$
\end{cor}

We are now ready to state how to enumerate triggers from the flood polynomial.

\begin{thm} \label{triggers}
 If $\displaystyle \F = \sum_{k=0}^n c_kx^k$, then the total number of triggers is \\ $\displaystyle \binom{n}{2} - (n-1)(n - c_{n-1}) + \binom{n - c_{n-1}}{2} - c_{n-2}$. 
\end{thm}

\begin{proof}

The structure of this proof will be very similar to that of Theorem \ref{numleaves}.  Let $T$ be the number of triggers in $G$.  
We will show $T + (n-1)(n - c_{n-1}) - \binom{n - c_{n-1}}{2} = \binom{n}{2} - c_{n-2}$.  Note that the right hand side, $\binom{n}{2} - c_{n-2}$, is the number of $(n-2)$-element non-flooding cascade sets.  Also note that by Theorem \ref{numleaves} and the property of inclusion-exclusion, $(n-1)(n - c_{n-1}) - \binom{n - c_{n-1}}{2}$ is the number of $(n-2)$-element cascade sets $C$ such that $V-C$ contains at least one leaf or isolated point.  Note that for all these sets, $V-C$ is not a trigger because the vertices that make up triggers have degree $2$.
We have already established by Corollary \ref{leaves in cascade set} that these cascade sets are non-flooding.  We also established by Lemma \ref{trigger flood} that the only other $(n-2)$-element non-flooding cascade sets are in bijection with the set of triggers.

Therefore we have that $T + (n-1)(n - c_{n-1}) - \binom{n - c_{n-1}}{2} = \binom{n}{2} - c_{n-2}$ as desired.
\end{proof}
%

\begin{cor}
    The total number of triggers of $G$ is determined by $\F$.
\end{cor}


As seen in this section, the values of $c_n$, $c_{n-1}$, and $c_{n-2}$ can be used to determine properties about the graph.  
It is still not known whether the value of any other coefficients can be used to determined the number of other subgraphs of $G$ similar to what we saw in Theorem \ref{numleaves} and Theorem \ref{triggers}.

\subsection{Free vertices}

We saw in the proof of Corollary \ref{leaves in cascade set} that if a graph contains any leaves or isolated points, then any flooding cascade set must contain all of those vertices.  In particular, the minimal flooding cascade sets must contain all of the leaves and isolated points.
On the other hand, it is possible that the inclusion of a vertex in a cascade set is never necessary to flood the graph. 
We say a vertex $v$ is \emph{free} if it is not an element of any minimal flooding cascade sets.

\begin{prop}
    If a vertex $v$ is neighbors with two or more leaves, then $v$ is free.
\end{prop}

\begin{proof}
    Suppose $v$ is a vertex with at least two leaves as neighbors and suppose that $v$ is an element of the flooding cascade set $C$.  We will prove this result by showing that $C$ is not a minimal flooding cascade set.  By Corollary \ref{leaves in cascade set}, we know that $C$ contains all of the leaves of the graph.  In particular, $C$ contains at least two neighbors of $v$.  Let $C' = C - \{v\}$.  Since $C'$ contains at least two neighbors of $v$, we have that $v\in C_1'$ so $C \subseteq C_1'$. By Proposition \ref{flooding superset}, it follows that $C_1' \in \FG$ so $C' \in \FG$.  Therefore $C$ is not minimal since  $C' \subsetneq C$, as desired.
\end{proof}

We just saw that any vertex with two or more leaves as neighbors is necessarily free, however, there are other, less trivial ways for a vertex to be free.

\begin{example} \label{free vertex example}

The following graph has two free vertices, but neither of them has two leaves as neighbors.  The free vertices are colored black. 

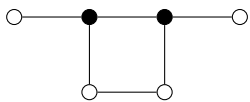
\begin{center}
\begin{tikzpicture}
\node (v0) at ( -1,1) [vertex] {};
\node (v1) at ( 0,0) [vertex] {};
\node (v2) at ( 0,1) [vertex, fill = black] {};
\node (v3) at ( 1,0) [vertex] {};
\node (v4) at ( 1,1) [vertex, fill = black] {};
\node (v5) at ( 2,1) [vertex] {};


\draw [-] (v0) to (v2);
\draw [-] (v1) to (v2);
\draw [-] (v1) to (v3);
\draw [-] (v2) to (v4);
\draw [-] (v3) to (v4);
\draw [-] (v4) to (v5);

\end{tikzpicture}
\end{center}

Observe that neither of those two vertices appear in the the two minimal flooding cascade sets shown below.

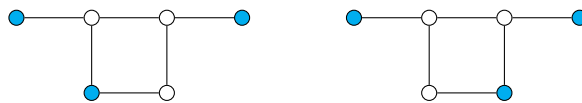
\begin{center}
\begin{tikzpicture}
\node (v0) at ( -1,1) [flood] {};
\node (v1) at ( 0,0) [flood] {};
\node (v2) at ( 0,1) [vertex] {};
\node (v3) at ( 1,0) [vertex] {};
\node (v4) at ( 1,1) [vertex] {};
\node (v5) at ( 2,1) [flood] {};


\draw [-] (v0) to (v2);
\draw [-] (v1) to (v2);
\draw [-] (v1) to (v3);
\draw [-] (v2) to (v4);
\draw [-] (v3) to (v4);
\draw [-] (v4) to (v5);

\end{tikzpicture}
\hspace{1 cm}
\begin{tikzpicture}
\node (v0) at ( -1,1) [flood] {};
\node (v1) at ( 0,0) [vertex] {};
\node (v2) at ( 0,1) [vertex] {};
\node (v3) at ( 1,0) [flood] {};
\node (v4) at ( 1,1) [vertex] {};
\node (v5) at ( 2,1) [flood] {};


\draw [-] (v0) to (v2);
\draw [-] (v1) to (v2);
\draw [-] (v1) to (v3);
\draw [-] (v2) to (v4);
\draw [-] (v3) to (v4);
\draw [-] (v4) to (v5);

\end{tikzpicture}
\end{center}

\end{example}

It is trivial that $\F$ gives an upper bound for the number of free vertices of $G$ since $\F$ gives the total number of vertices.  However, a stronger upper bound can be determined from $\F$.

\begin{thm} \label{free vertex theorem}
    The number free vertices of $G$ is bounded above by the number of factors of $(x+1)$ in $\F$. 
\end{thm}

\begin{proof}

Let $\mathcal{S}$ be the set of free vertices of $G$, $s = |\mathcal{S}|$, and let $\mathcal{P}(\mathcal{S})$ be the power set of $\mathcal{S}$. We want to show that $(x+1)$ is a factor of $\F$ with degree at least $s$.  Let $\mathcal{M} = \{C \in \FG \mid C \cap \mathcal{S} = \varnothing\}$, i.e., $\mathcal{M}$ is the set of all flooding cascade sets of $G$ that contain no free vertices.   By definition of a free vertex, every minimal flooding cascade set is an element of $\mathcal{M}$.  Let $\mathcal{A} = \{M \cup S \mid M \in \mathcal{M} \text{ and } S \in \mathcal{P}(\mathcal{S})\}$.  We will show that $\FG = \mathcal{A}$.   

Let $M \in \mathcal{M}$ and $S \in \mathcal{P}(\mathcal{S})$.  Since every element of $\mathcal{M}$ is a flooding cascade set, then by Proposition \ref{flooding superset}, we have that $(M \cup S) \in \FG$.  Therefore $\mathcal{A} \subseteq \FG$.

If $C \in \FG$, then to show $C \in \mathcal{A}$, we need to show that $(C - \mathcal{S}) \in \mathcal{M}$.  Since $C \in \FG$ there exists a minimal flooding cascade set $M$ such that $M \subseteq C$.  It follows that $M = (M-\mathcal{S}) \subseteq(C-\mathcal{S})$.  Therefore $C-\mathcal{S}$ is a flooding cascade set that contains no free vertices.  Hence $C \in \mathcal{A}$ and $\FG \subseteq \mathcal{A}$.

It follows that 
\[\F = \sum_{\substack{M \in \mathcal{M}, \\ S\in \mathcal{P{(\mathcal{S})}}}} x^{|M|+|S|}= \sum_{M \in \mathcal{M}}x^{|M|}\cdot \sum_{S \in \mathcal{P(S)}}x^{|S|} = \sum_{M \in \mathcal{M}}x^{|M|}(x+1)^s\]
as desired.
\end{proof}
The graph from Example \ref{free vertex example} has a flood polynomial that factors as $x^3(x+2)(x+1)^2$ and it has $2$ free vertices.  In general, the inequality given in Theorem \ref{free vertex theorem} will not reduce to an equality.  In fact, the following example shows that the number of free vertices is not determined by the flood polynomial.
%
%

\begin{example} \label{free vertices example}
The following graphs have $x^4(x+1)(x+3)$ as their flood polynomial, but a different number of free vertices. The free vertex is colored black.

\begin{center}
\begin{tikzpicture}
\node at (-1, 0) {$P_6:$};
\node (v1) at ( 0,0) [vertex] {};
\node (v2) at ( 1,0) [vertex] {};
\node (v3) at ( 2,0) [vertex] {};
\node (v4) at ( 3,0) [vertex] {};
\node (v5) at ( 4,0) [vertex] {};
\node (v6) at ( 5,0) [vertex] {};
\draw [-] (v1) to (v2);
\draw [-] (v2) to (v3);
\draw [-] (v3) to (v4);
\draw [-] (v4) to (v5);
\draw [-] (v5) to (v6);
\end{tikzpicture}
\end{center}
\vspace{.5 cm}
\begin{center}
\begin{tikzpicture}
 \node at (-1.5, .45) {$P_3 \oplus C_3:$};
\node (v1) at ( 0,.45) [vertex] {};
\node (v2) at ( 1,.45) [vertex, fill=black] {};
\node (v3) at ( 2,.45) [vertex] {};
\node (v4) at ( 3,0) [vertex] {};
\node (v5) at ( 4,0) [vertex] {};
\node (v6) at ( 3.5,.9) [vertex] {};
\draw [-] (v1) to (v2);
\draw [-] (v2) to (v3);
\draw [-] (v4) to (v6);
\draw [-] (v4) to (v5);
\draw [-] (v5) to (v6);
\end{tikzpicture}
\end{center}

Even though $F_{P_6}(x)$ has a factor of $(x+1)$, $P_6$ has no free vertices 
as you can see by the minimal flooding cascade sets shown below.

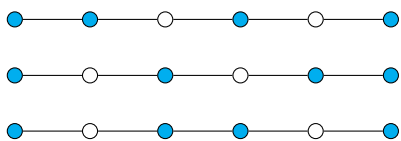
\begin{center}
\begin{tikzpicture}
\node (v1) at ( 0,0) [flood] {};
\node (v2) at ( 1,0) [flood] {};
\node (v3) at ( 2,0) [vertex] {};
\node (v4) at ( 3,0) [flood] {};
\node (v5) at ( 4,0) [vertex] {};
\node (v6) at ( 5,0) [flood] {};

\draw [-] (v1) to (v2);
\draw [-] (v2) to (v3);
\draw [-] (v3) to (v4);
\draw [-] (v4) to (v5);
\draw [-] (v5) to (v6);

\end{tikzpicture}

\vspace{.25 cm}

\begin{tikzpicture}
\node (v1) at ( 0,0) [flood] {};
\node (v2) at ( 1,0) [vertex] {};
\node (v3) at ( 2,0) [flood] {};
\node (v4) at ( 3,0) [vertex] {};
\node (v5) at ( 4,0) [flood] {};
\node (v6) at ( 5,0) [flood] {};

\draw [-] (v1) to (v2);
\draw [-] (v2) to (v3);
\draw [-] (v3) to (v4);
\draw [-] (v4) to (v5);
\draw [-] (v5) to (v6);

\end{tikzpicture}

\vspace{.25 cm}

\begin{tikzpicture}
\node (v1) at ( 0,0) [flood] {};
\node (v2) at ( 1,0) [vertex] {};
\node (v3) at ( 2,0) [flood] {};
\node (v4) at ( 3,0) [flood] {};
\node (v5) at ( 4,0) [vertex] {};
\node (v6) at ( 5,0) [flood] {};

\draw [-] (v1) to (v2);
\draw [-] (v2) to (v3);
\draw [-] (v3) to (v4);
\draw [-] (v4) to (v5);
\draw [-] (v5) to (v6);

\end{tikzpicture}
\end{center}

\end{example}

\section{Flood Polynomials for Families of Graphs}

In this section we state a formula for the flood polynomials for three families of graphs: parallel paths, cycles, and triangle mosaics.

\subsection{Parallel Paths} 

A \emph{parallel path graph of size $m \times n$} is a graph with $m \cdot n$ vertices, denoted $P_{m, n}$. It is the graph with vertex set \[V(P_{m, n}) = \{v_{i, j} \mid 1 \leq i \leq m \text{ and } 1 \leq j \leq n\}\] 
and edge set 
\[E(P_{m, n}) = \{v_{i, j}v_{i, j+1} \mid 1 \leq j < n\} \cup \{v_{i, j}v_{i+1, j} \mid 1 \leq i < m\}.\]

It follows immediately from the construction of parallel path graphs that $P_{m, n} \cong P_{n, m}$.  When $m=1$, the graph is a path and we denote $P_{1, n}$ by $P_n$.  
As an example, the following eight-element graph is the parallel path graph of size $2 \times 4$. Our convention will be to label the vertices as you would label the entries in a matrix.  
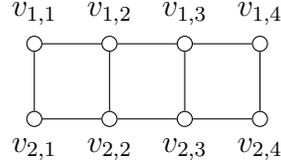
\begin{center}
\begin{tikzpicture}
    \node (v1) at (-4,0) [vertex, label = above: $v_{1,1}$]  {}; 
    \node (v2) at (-4,-1) [vertex, label = below: $v_{2,1}$] {};
    \node (v3) at (-3,0) [vertex, label = above: $v_{1,2}$] {};
    \node (v4) at (-3,-1) [vertex, label = below: $v_{2,2}$] {};
    \node (v5) at (-2,0) [vertex, label = above: $v_{1,3}$] {};
    \node (v6) at (-2,-1) [vertex, label = below: $v_{2,3}$] {};
    \node (v7) at (-1,0) [vertex, label = above: $v_{1,4}$] {};
    \node (v8) at (-1,-1) [vertex, label = below: $v_{2,4}$] {};

    \draw [-] (v1) to (v3);
    \draw [-] (v1) to (v2); 
    \draw [-] (v2) to (v4);
    \draw [-] (v3) to (v5);
    \draw [-] (v3) to (v4);
    \draw [-] (v4) to (v6);
    \draw [-] (v5) to (v6);
    \draw [-] (v6) to (v8);
    \draw [-] (v7) to (v8);
    \draw [-] (v5) to (v7);
\end{tikzpicture}
\end{center}

We now discuss properties of the flooding cascade sets of parallel path graphs. These are useful for giving a recursion for the flood polynomial of these graphs.  

\begin{lemma} \label{parallel path empty column}
    If $C \in \mathcal{F}(P_{m, n})$, then $\{v_{1,1}, \dots v_{m, 1}\} \cap C \neq \varnothing $ and $\{v_{1,n}, \dots v_{m, n}\} \cap C \neq \varnothing $. 
\end{lemma}

\begin{proof}
    We will show that $\{v_{1,1}, \dots v_{m, 1}\} \cap C \neq \varnothing $ and by symmetry we can conclude that $\{v_{1,n}, \dots v_{m, n}\} \cap C \neq \varnothing $.

    Suppose for contradiction that $C$ completely floods $P_{m, n}$, but does not contain any element from $\{v_{1,1}, \dots v_{m, 1}\}$, i.e., $\{v_{1,1}, \dots v_{m, 1}\} \cap C = \varnothing $.  Since $C$ completely floods $P_{m, n}$, there must eventually be a term in the cascade sequence that contains an element of $\{v_{1,1}, \dots v_{m, 1}\}$.  Let $i$ be the smallest number such that $\{v_{1,1}, \dots v_{m, 1}\} \cap C_i = \varnothing $, but $\{v_{1,1}, \dots v_{m, 1}\} \cap C_{i+1} \neq \varnothing$.  Since the intersection is nonempty, there is a value $j$ such that $v_{j,1} \in C_{i+1}$. In order for $v_{j,1} \in C_{i+1}$, but $v_{j,1} \notin C_{i}$, it must be the case that at least two of the neighbors of $v_{j, 1}$ are in $C_i$.  The neighbors of $v_{j, 1}$ are $v_{j+1,1}$, $v_{j-1,1}$, and $v_{j,2}$ (some of these vertices do not exist if $j=1$, $j=m$, or $n=1$). Therefore, it cannot be the case that at least two of these neighbors are in $C_i$ because $\{v_{1,1}, \dots v_{m, 1}\} \cap C_i = \varnothing $.  Therefore, there is no such $C$ that completely floods $P_{m, n}$, as desired.
\end{proof}

A similar argument can be made to prove the following.

\begin{cor}
        If $C \in \mathcal{F}(P_{m, n})$, then $\{v_{1,1}, \dots v_{1, n}\} \cap C \neq \varnothing $ and $\{v_{m,1}, \dots v_{m, n}\} \cap C \neq \varnothing $. 
\end{cor}

Informally, the previous Lemma and Corollary state that all flooding cascade sets of $P_{m, n}$ contain at least one vertex from both the first column and last column of $P_{m, n}$, as well as at least one vertex from both the top and bottom row of $P_{m, n}$.  

\begin{lemma} \label{2 column gap parallel paths}
    If $C \in \mathcal{F}(P_{m, n})$, then for all $1 < l < n$, \\
   $(\{v_{1,l}, \dots v_{m, l}\} \cup \{v_{1,l+1}, \dots v_{m, l+1}\}) \cap C \neq \varnothing.$
\end{lemma}

\begin{proof}
    Let $1 < l < n$ be arbitrary and as in the proof of Lemma \ref{parallel path empty column}, we will suppose for contradiction that $C$ completely floods $P_{m, n}$, but does not contain any elements from $(\{v_{1,l}, \dots v_{m, l}\} \cup \{v_{1,l+1}, \dots v_{m, l+1}\})$.  
    Let $V_l = (\{v_{1,l}, \dots v_{m, l}\} \cup \{v_{1,l+1}, \dots v_{m, l+1}\})$.
    Since $C$ completely floods $P_{m, n}$, there must eventually be a term in the cascade sequence that contains an element of $V_l$.
    Let $i$ be the smallest number such that $V_l \cap C_i = \varnothing $, but $V_l \cap C_{i+1} \neq \varnothing$.  Since the intersection is nonempty, there is a value $j$ such that either $v_{j,l} \in C_{i+1}$ or $v_{j,l+1} \in C_{i+1}$.
    Without loss of generality, we can assume that $v_{j,l} \in C_{i+1}$.  The neighbors of $v_{j,l}$ are $v_{j+1,l}$, $v_{j-1,l}$, $v_{j,l-1}$, and $v_{j,l+1}$ (some of these vertices do not exist if $j = 1$ or  $j=m$). 
    Therefore, it cannot be the case that at least two of these neighbors are in $C_i$ because $V_l \cap C_i = \varnothing $ and $\{v_{j+1,l}, v_{j-1,l},  v_{j,l+1}\} \subseteq V_l$. Therefore, there is no such $C$ that completely floods $P_{m, n}$, as desired.
\end{proof}

We say that any cascade set of $P_{m, n}$ that satisfies the conclusions of Lemma \ref{parallel path empty column} and Lemma \ref{2 column gap parallel paths} has the \emph{parallel path property}.  Note that a cascade set does not need to be a flooding cascade set in order to have the parallel path property.

We now state recursive formulas for the flood polynomials of $P_{n}$ and $P_{2, n}$, followed by a discussion about the difficulty in finding a recursive formula for $P_{m, n}$ when $m \geq 3$. 




\subsubsection{$P_n$} \label{Pn}

While the parallel path property does not guarantee a cascade set of $P_{m, n}$ is a flooding cascade set, we will see that it does imply this result for $P_n$.

\begin{thm} \label{PPP Path}
    Let $C$ be a cascade set of $P_n$.  Then $C \in \mathcal{F}(P_n)$ if and only if $C$ has the parallel path property.
\end{thm}

\begin{proof}
    If $C \in \mathcal{F}(P_n)$, then by Lemma \ref{parallel path empty column} and Lemma \ref{2 column gap parallel paths} we have that $C$ has the parallel path property.

    Now suppose that $C$ is a cascade set with the parallel path property, i.e., with $\{v_1, v_n\}\subseteq C$ and for all $1<i<n$, $\{v_i, v_{i+1}\} \nsubseteq (V-C)$.  Let $v_j$ be an arbitrary element in $V - C$.  Since $\{v_1, v_n\}\subseteq C$, $j \notin \{1, n\}.$  Since $\{v_{j-1}, v_{j}\} \nsubseteq (V-C)$ and $\{v_j, v_{j+1}\} \nsubseteq (V-C)$, we have that $\{v_{j-1}, v_{j+1}\} \subseteq C$.  Therefore, $v_j \in C_1$ and $C_1 = V$.  Therefore, $C \in \mathcal{F}(P_n)$ as desired.
\end{proof}

We can now give a recursion for the flood polynomial of path graphs.

\begin{thm} \label{path polynomial}
The flood polynomial for path graphs $P_n$ can be determine recursively with $F_{P_1}(x) = x$, $F_{P_2}(x) = x^2$, and for $n \geq 3$, $F_{P_n}(x) = x\cdot F_{P_{n-1}}(x)+ x \cdot F_{P_{n-2}}(x)$.
\end{thm}

\begin{proof}
    It is easy to see that $F_{P_1}(x) = x$ and $F_{P_2}(x) = x^2$, so our initial conditions hold.

    Now let $n \geq 3.$  We will show that  $F_{P_n}(x) = x\cdot F_{P_{n-1}}(x)+ x \cdot F_{P_{n-2}}(x)$ by showing that each element of $\mathcal{F}(P_n)$ can be expressed as the union of $\{v_n\}$ with either an element of $\mathcal{F}(P_{n-1})$ or $\mathcal{F}(P_{n-2})$.

    Let $C \in \mathcal{F}(P_n)$ and let $C' = C - \{v_n\}$.  By Theorem \ref{PPP Path}, we have that $v_n \in C$ and hence $|C'| = |C| - 1$.  That is to say $x^{|C|} = x\cdot x^{|C'|}$.  It is also the case that at least one of $v_{n-2}$ or  $v_{n-1}$ is in $C$ and hence, $C'$.  

    If $v_{n-1} \in C'$, then $C'$ has the parallel path property when viewed as a cascade set of $P_{n-1}$. Therefore $C' \in \mathcal{F}(P_{n-1})$.  Note that every flooding cascade set of $P_{n-1}$ will be considered in this steps because if $C \in \mathcal{F}(P_{n-1})$, then $(C \cup \{v_n\}) \in \mathcal{F}(P_{n})$.

    If $v_{n-1} \notin C'$, then it must be the case that $v_{n-2} \in C'$.  Therefore $C'$ has the parallel path property when viewed as a cascade set of $P_{n-2}$. Therefore $C' \in \mathcal{F}(P_{n-2})$.  Note that every flooding cascade set of $P_{n-2}$ will be considered in this steps because if $C \in \mathcal{F}(P_{n-2})$, then $(C \cup \{v_n\}) \in \mathcal{F}(P_{n})$.

    Therefore we have that $F_{P_n}(x) = x\cdot F_{P_{n-1}}(x)+ x \cdot F_{P_{n-2}}(x)$ as desired.
\end{proof}

The \emph{Fibonacci numbers} are a sequence of numbers defined recursively by $f_0 = 0$, $f_1 = 1$, and for $n \geq 2$, $f_n = f_{n-1}+f_{n-2}$; the first few Fibonacci numbers are $0, 1, 1, 2, 3, 5, \dots$.  The Fibonacci numbers count numerous combinatorial objects 
\cite{OEIS_A000045}, but in particular, it is well-known that for $n\geq 2$, $F_n$ is the number of binary sequences of length $n-2$ that have no consecutive $0$'s, as well as the number of binary sequences of length $n$ with no no consecutive $0$'s whose first and last entry is $1$. This set is in bijection with the flood set of $P_n$ under the map $C \leftrightarrow \sigma$ where $\sigma_i = 1$ if $v_i \in C$ and $\sigma_i = 0$ otherwise.




\begin{cor} \label{fibonacci numbers}
For all $n \geq 1$, the total number of flooding cascade sets of $P_n$ is equal to the $n$th Fibonacci number, i.e., $F_{P_n}(1) = f_n$.
\end{cor}

There are many ways to define the \emph{Fibonacci polynomials}, but one such definition is they are the polynomials defined recursively as $f_0(x) = 0$, $f_1(x) = x$, and for $n \geq 2$, $f_n(x) = x\cdot f_{n-1}(x)+  x\cdot f_{n-2}(x)$ (see \cite{EoM_FibPolynomials}). It follows immediately from Theorem \ref{path polynomial} that $F_{P_n}(x) = f_n(x)$.

We now give a combinatorial interpretation of the coefficients of $f_n(x)$ in terms of flooding cascade sets of $P_n$.


\begin{cor} \label{fibonacci polynomial}
    If $f_n(x) = \sum_{k=0}^nf(n, k)x^k$, then $f(n, k)$ is the number $k$-element flooding cascade sets of $P_n$. 
\end{cor}

Many classical interpretations of the coefficients of the Fibonacci polynomials involve enumerating objects of size less than $n$ (see \cite{OEIS_A011973}).  An advantage of the combinatorial interpretations of both the $n$th Fibonacci number given in Corollary \ref{fibonacci numbers} and the coefficients of the $n$th Fibonacci polynomial given in Corollary \ref{fibonacci polynomial} is that it comes from a graph with $n$ vertices. 

\subsubsection{$P_{2, n}$}

The remainder of this subsection provides results regarding the parallel path graphs of size $2 \times n$.  

\begin{lemma} \label{2xn flooding conditions}
    Suppose that $C$ is a cascade set of $G = P_{2, n}$ that has the parallel path property. Then $C \in \FG$ if any only if for some $v_{1, j} \in C$, we have $\{v_{2, j-1}, v_{2, j}, v_{2, j+1}\}\cap C \neq \varnothing$.
\end{lemma}

\begin{proof}
    Suppose that $C$ is a cascade set of $G = P_{2, n}$ that has the parallel path property and that for some $v_{1, j} \in C$, we have $\{v_{2, j-1}, v_{2, j}, v_{2, j+1}\}\cap C \neq \varnothing$, we want to show that $C \in \FG$.  We will prove this by inducting on $n$.

    If $n=1$, then $G = P_{2, 1}$.  The only cascade set that satisfies the hypotheses is $\{v_{1, 1}, v_{2, 1}\}$ and this floods $P_{2, 1}$.

    Now suppose $n > 1$ and the result holds for all $P_{2, k}$ where $1 \leq k < n$.  Let $C$ be a cascade set that has the parallel path property and that for some $v_{1, j} \in C$, we have $\{v_{2, j-1}, v_{2, j}, v_{2, j+1}\}\cap C \neq \varnothing$.  

    If $v_{1, 1} \in C$, and $\{v_{2, 1}, v_{2, 2}\}\cap C \neq \varnothing$, then $\{v_{1, 2}, v_{2, 2}\} \subseteq C_2$ because $\{v_{1, 2}, v_{2, 2}, v_{1, 3}, v_{2, 3}\} \cap C \neq \varnothing$. 

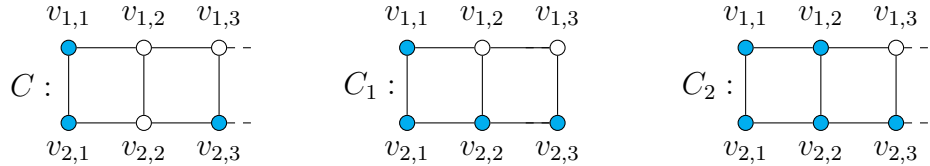
    \begin{center}
\begin{tikzpicture}
    \node at (-.5, .5) {$C:$};
    \node (v1) at (0,1) [flood, label = above: $v_{1,1}$]  {}; 
    \node (v2) at (0,0) [flood, label = below: $v_{2,1}$] {};
    \node (v3) at (1,1) [vertex, label = above: $v_{1,2}$] {};
    \node (v4) at (1,0) [vertex, label = below: $v_{2,2}$] {};
    \node (v5) at (2,1) [vertex, label = above: $v_{1,3}$] {};
    \node (v6) at (2,0) [flood, label = below: $v_{2,3}$] {};

    \draw [-] (v1) to (v3);
    \draw [-] (v1) to (v2); 
    \draw [-] (v2) to (v4);
    \draw [-] (v3) to (v5);
    \draw [-] (v3) to (v4);
    \draw [-] (v4) to (v6);
    \draw [-] (v5) to (v6);
    \draw [-, dashed] (v5) to (2.5, 1);
    \draw [-, dashed] (v6) to (2.5, 0);

    \node at (4, .5) {$C_1:$};
    \node (v1) at (4.5,1) [flood, label = above: $v_{1,1}$]  {}; 
    \node (v2) at (4.5,0) [flood, label = below: $v_{2,1}$] {};
    \node (v3) at (5.5,1) [vertex, label = above: $v_{1,2}$] {};
    \node (v4) at (5.5,0) [flood, label = below: $v_{2,2}$] {};
    \node (v5) at (6.5,1) [vertex, label = above: $v_{1,3}$] {};
    \node (v6) at (6.5,0) [flood, label = below: $v_{2,3}$] {};

    \draw [-] (v1) to (v3);
    \draw [-] (v1) to (v2); 
    \draw [-] (v2) to (v4);
    \draw [-] (v3) to (v5);
    \draw [-] (v3) to (v4);
    \draw [-] (v4) to (v6);
    \draw [-] (v5) to (v6);
    \draw [-, dashed] (v5) to (6, 1);
    \draw [-, dashed] (v6) to (6, 0);

    \node at (8.5, .5) {$C_2:$};
    \node (v1) at (9,1) [flood, label = above: $v_{1,1}$]  {}; 
    \node (v2) at (9,0) [flood, label = below: $v_{2,1}$] {};
    \node (v3) at (10,1) [flood, label = above: $v_{1,2}$] {};
    \node (v4) at (10,0) [flood, label = below: $v_{2,2}$] {};
    \node (v5) at (11,1) [vertex, label = above: $v_{1,3}$] {};
    \node (v6) at (11,0) [flood, label = below: $v_{2,3}$] {};

    \draw [-] (v1) to (v3);
    \draw [-] (v1) to (v2); 
    \draw [-] (v2) to (v4);
    \draw [-] (v3) to (v5);
    \draw [-] (v3) to (v4);
    \draw [-] (v4) to (v6);
    \draw [-] (v5) to (v6);
    \draw [-, dashed] (v5) to (11.5, 1);
    \draw [-, dashed] (v6) to (11.5, 0);
\end{tikzpicture}
\end{center}

    The subgraph consisting of all vertices except $\{v_{1, 1}, v_{2, 1}\}$ is isomorphic to $P_{2, {n-1}}$ and by induction the subset of $C_2$ consisting of only the vertices of $C_2$ in this subgraph floods this subgraph.  Therefore $C \in \FG$.

    A similar argument can be made to show that if $v_{1, n} \in C$, and $\{v_{2, n-1}, v_{2, n}\}\cap C \neq \varnothing$, then $C \in \FG$.

    If $v_{1, j} \in C$ and $\{v_{2, j-1}, v_{2, j}, v_{2, j+1}\}\cap C \neq \varnothing$ for some $1<j<n$, then $\{v_{1, j}, v_{2, j}\} \subseteq C_1$.  By induction $C_1$ floods both the subgraph consisting of vertices $\{v_{1, 1}, v_{2, 1}, \dots, v_{1, j}, v_{2, j}\}$ and the subgraph consisting of vertices $\{v_{1, j}, v_{2, j}, \dots, v_{1, n}, v_{2, n}\}$. Therefore $C \in \FG$.

    Now suppose that $C$ has the parallel path property and that for all $j$, whenever $v_{1, j} \in C$ we have that $\{v_{2, j-1}, v_{2, j}, v_{2, j+1}\} \cap C = \varnothing$.
    By Lemma \ref{parallel path empty column} and our hypothesis, we know that either $v_{1, 1} \in C$ or $v_{2, 1} \in C$, but not both.  By symmetry, we can assume that $v_{1, 1} \in C$.  Let $k_1$ be the smallest number such that $v_{2, k_1} \in C$. It is possible that no such $k_1$ exists, but if it does, then $k_1 \geq 3$.  If $k_1$ exists, let $k_2$ be the smallest number such that $k_1 < k_2$ and $v_{1, k_2} \in C$.  Note that if $k_2$ exists, then $k_2 \geq k_1+2$.  Define $k_l$ similarly so that $k_{l-1}<k_l$ and $v_{(l \textrm{ mod } 2) + 1, k_l} \in C$.  Let $L$ be the largest number such that $k_L$ is defined and let $a = 1$ if $L$ is even and let $a=2$ if $L$ is odd.

    It follows that $C_1 = \{v_{1, 1} \dots v_{1, k_{1}-2}\} \cup \{v_{2, k_l} \dots v_{2, k_{2}-2}\} \cup \dots \cup \{v_{a, k_L} \dots v_{a, n}\}$.

        \begin{center}
\begin{tikzpicture}
    \node at (-.5, .5) {$C:$};
    \node (v11) at (0,1) [flood, label = above: $v_{1,1}$]  {}; 
    \node (v12) at (0,0) [vertex] {};
    \node (v13) at (1,1) [flood] {};
    \node (v14) at (1,0) [vertex] {};
    \node (v15) at (2,1) [vertex] {};
    \node (v16) at (2,0) [vertex] {};

    \node (v21) at (3.5,1) [flood, label = above: $v_{1,k_1-2}$]  {}; 
    \node (v22) at (3.5,0) [vertex] {};
    \node (v23) at (4.5,1) [vertex] {};
    \node (v24) at (4.5,0) [vertex] {};
    \node (v25) at (5.5,1) [vertex] {};
    \node (v26) at (5.5,0) [flood, label = below: $v_{2,k_1}$] {};
    \node (v27) at (6.5,1) [vertex] {};
    \node (v28) at (6.5,0) [vertex] {};
    \node (v29) at (7.5,1) [vertex] {};
    \node (v210) at (7.5,0) [flood] {};

    \node (v31) at (9,1) [vertex]  {}; 
    \node (v32) at (9,0) [flood, label = below: $v_{2,k_2-2}$] {};
    \node (v33) at (10,1) [vertex] {};
    \node (v34) at (10,0) [vertex] {};
    \node (v35) at (11,1) [flood, label = above: $v_{1,k_2}$] {};
    \node (v36) at (11,0) [vertex] {};
    \node (v37) at (12,1) [flood] {};
    \node (v38) at (12,0) [vertex] {};
    \node (v39) at (13,1) [flood] {};
    \node (v310) at (13,0) [vertex] {};

    \draw [-] (v11) to (v13);
    \draw [-] (v11) to (v12); 
    \draw [-] (v12) to (v14);
    \draw [-] (v13) to (v15);
    \draw [-] (v13) to (v14);
    \draw [-] (v14) to (v16);
    \draw [-] (v15) to (v16);
    \draw [-, dashed] (v15) to (2.5, 1);
    \draw [-, dashed] (v16) to (2.5, 0);

    \draw [-, dashed] (v21) to (3, 1);
    \draw [-, dashed] (v22) to (3, 0);
    \draw [-] (v21) to (v23);
    \draw [-] (v21) to (v22); 
    \draw [-] (v22) to (v24);
    \draw [-] (v23) to (v25);
    \draw [-] (v23) to (v24);
    \draw [-] (v24) to (v26);
    \draw [-] (v25) to (v26);
    \draw [-] (v26) to (v28);
    \draw [-] (v27) to (v28);
    \draw [-] (v25) to (v27);
    \draw [-] (v28) to (v210);
    \draw [-] (v29) to (v210);
    \draw [-] (v27) to (v29);
    \draw [-, dashed] (v29) to (8, 1);
    \draw [-, dashed] (v210) to (8, 0);

    \draw [-, dashed] (v31) to (8.5, 1);
    \draw [-, dashed] (v32) to (8.5, 0);
    \draw [-] (v31) to (v33);
    \draw [-] (v31) to (v32); 
    \draw [-] (v32) to (v34);
    \draw [-] (v33) to (v35);
    \draw [-] (v33) to (v34);
    \draw [-] (v34) to (v36);
    \draw [-] (v35) to (v36);
    \draw [-] (v36) to (v38);
    \draw [-] (v37) to (v38);
    \draw [-] (v35) to (v37);
    \draw [-] (v38) to (v310);
    \draw [-] (v39) to (v310);
    \draw [-] (v37) to (v39);
    \draw [-, dashed] (v39) to (13.5, 1);
    \draw [-, dashed] (v310) to (13.5, 0);

\end{tikzpicture}

\begin{tikzpicture}
    \node at (-.5, .5) {$C_1:$};
    \node (v11) at (0,1) [flood, label = above: $v_{1,1}$]  {}; 
    \node (v12) at (0,0) [vertex] {};
    \node (v13) at (1,1) [flood] {};
    \node (v14) at (1,0) [vertex] {};
    \node (v15) at (2,1) [flood] {};
    \node (v16) at (2,0) [vertex] {};

    \node (v21) at (3.5,1) [flood, label = above: $v_{1,k_1-2}$]  {}; 
    \node (v22) at (3.5,0) [vertex] {};
    \node (v23) at (4.5,1) [vertex] {};
    \node (v24) at (4.5,0) [vertex] {};
    \node (v25) at (5.5,1) [vertex] {};
    \node (v26) at (5.5,0) [flood, label = below: $v_{2,k_1}$] {};
    \node (v27) at (6.5,1) [vertex] {};
    \node (v28) at (6.5,0) [flood] {};
    \node (v29) at (7.5,1) [vertex] {};
    \node (v210) at (7.5,0) [flood] {};

    \node (v31) at (9,1) [vertex]  {}; 
    \node (v32) at (9,0) [flood, label = below: $v_{2,k_2-2}$] {};
    \node (v33) at (10,1) [vertex] {};
    \node (v34) at (10,0) [vertex] {};
    \node (v35) at (11,1) [flood, label = above: $v_{1,k_2}$] {};
    \node (v36) at (11,0) [vertex] {};
    \node (v37) at (12,1) [flood] {};
    \node (v38) at (12,0) [vertex] {};
    \node (v39) at (13,1) [flood] {};
    \node (v310) at (13,0) [vertex] {};

    \draw [-] (v11) to (v13);
    \draw [-] (v11) to (v12); 
    \draw [-] (v12) to (v14);
    \draw [-] (v13) to (v15);
    \draw [-] (v13) to (v14);
    \draw [-] (v14) to (v16);
    \draw [-] (v15) to (v16);
    \draw [-, dashed] (v15) to (2.5, 1);
    \draw [-, dashed] (v16) to (2.5, 0);

    \draw [-, dashed] (v21) to (3, 1);
    \draw [-, dashed] (v22) to (3, 0);
    \draw [-] (v21) to (v23);
    \draw [-] (v21) to (v22); 
    \draw [-] (v22) to (v24);
    \draw [-] (v23) to (v25);
    \draw [-] (v23) to (v24);
    \draw [-] (v24) to (v26);
    \draw [-] (v25) to (v26);
    \draw [-] (v26) to (v28);
    \draw [-] (v27) to (v28);
    \draw [-] (v25) to (v27);
    \draw [-] (v28) to (v210);
    \draw [-] (v29) to (v210);
    \draw [-] (v27) to (v29);
    \draw [-, dashed] (v29) to (8, 1);
    \draw [-, dashed] (v210) to (8, 0);

    \draw [-, dashed] (v31) to (8.5, 1);
    \draw [-, dashed] (v32) to (8.5, 0);
    \draw [-] (v31) to (v33);
    \draw [-] (v31) to (v32); 
    \draw [-] (v32) to (v34);
    \draw [-] (v33) to (v35);
    \draw [-] (v33) to (v34);
    \draw [-] (v34) to (v36);
    \draw [-] (v35) to (v36);
    \draw [-] (v36) to (v38);
    \draw [-] (v37) to (v38);
    \draw [-] (v35) to (v37);
    \draw [-] (v38) to (v310);
    \draw [-] (v39) to (v310);
    \draw [-] (v37) to (v39);
    \draw [-, dashed] (v39) to (13.5, 1);
    \draw [-, dashed] (v310) to (13.5, 0);

\end{tikzpicture}

\end{center}

    There are no elements of $V(P_{2, n})- C_1$ that are neighbors with two elements of $C_1$.  Therefore $\C = C_1$ and $C \notin \FG$ as desired.
\end{proof}


We showed in Lemma \ref{parallel path empty column} and Lemma \ref{2 column gap parallel paths} that in order for a cascade set to be an element of $\mathcal{F}(P_{2, n})$ it must have the parallel path property.  Therefore the set of flooding cascade sets of $P_{2, n}$ is equal to the number of flooding cascade sets of $P_{2, n}$ that have the parallel path property.  This fact will be used with the following two lemmas to give a recursion for the flood polynomial of $P_{2, n}$.

\begin{lemma} \label{A formula}
    Let $\mathcal{P}_n$ be the set of cascade sets of $P_{2, n}$ with the parallel path property and let $A_n(x)$ be the sequence of polynomials defined by $A_n(x) = \sum_{C \in \mathcal{P}_n}x^{|C|}$.  Then $A_1(x) = x^2+2x, A_2(x) = (x^2+2x)^2$, and for $n\geq 3$, $A_n(x) = (x^2+2x)(A_{n-1}(x)+A_{n-2}(x)).$
\end{lemma}

\begin{proof}
    The elements of $\mathcal{P}_1$ are $\{v_{1, 1}\}, \{v_{2, 1}\}$ and $\{v_{1, 1}, v_{2, 1}\}$.  Therefore $A_1(x) = x^2+2x$.
    The elements of $\mathcal{P}_2$ are the cascade sets shown in Example \ref{Flood Set} as well as $\{v_{1, 1}, v_{1, 2}\}$ and $ \{v_{2, 1}, v_{2, 2}\}$.  Therefore $A_2(x) = x^4+4x^3+4x^2 = (x^2+2x)^2$.

    We will now show that $A_n(x) = (x^2+2x)( \sum_{C \in \mathcal{P}_{n-1}}x^{|C|} + \sum_{C \in \mathcal{P}_{n-2}}x^{|C|})$ by showing that each element of $\mathcal{P}_n$ can be expressed as the union of a nonempty subset of $\{v_{1, n}, v_{2, n}\}$ with an element of $\mathcal{P}_{n-1}$ or $\mathcal{P}_{n-2}$.

    Let $C \in \mathcal{P}_n$ and let $\hat{C} = C \cap \{v_{1, n}, v_{2, n}\}$ and $C^* = C-\hat{C}$.  Informally speaking, $\hat{C}$ is the elements in $C$ that appear in the last column of $P_{2, n}$ and $C^*$ is the set of elements in $C$ that appear in the first $n-1$ columns of $P_{2, n}$.  By the parallel path property we know that $\hat{C} \neq \varnothing$.  This means that $\hat{C} = \{v_{1, n}\}$, $\hat{C} = \{v_{2, n}\}$, or $\hat{C} = \{v_{1, n}, v_{2, n}\}$ and none of these possibilities can cause $C$ to fail the parallel path property. If $C^* \cap \{v_{1, n-1}, v_{2, n-1}\} \neq \varnothing$, then $C^* \in \mathcal{P}_{n-1}$.  If $C^* \cap \{v_{1, n-1}, v_{2, n-1}\} = \varnothing$, then by the parallel path property, $C^* \cap \{v_{1, n-2}, v_{2, n-2}\} \neq \varnothing$ and hence $C^* \in \mathcal{P}_{n-2}$.  It is straightforward to reverse this map. Therefore $A_n(x) = (x^2+2x)( \sum_{C \in \mathcal{P}_{n-1}}x^{|C|} + \sum_{C \in \mathcal{P}_{n-2}}x^{|C|})$ as desired. 
\end{proof}

\begin{lemma} \label{B formula}
    Let $\tilde{\mathcal{P}}_n$ be the set of non-flooding cascade sets of $P_{2, n}$ with the parallel path property and let $B_n(x)$ be the sequence of polynomials defined by $ B_n(x) = \sum_{C \in \hat{\mathcal{P}}_n}x^{|C|}$. Then $B_1(x) = 2x, B_2(x) = 2x^2$, and for $n\geq 3$, $B_n(x) = x(B_{n-1}(x)+2\cdot B_{n-2}(x)).$
\end{lemma}

\begin{proof}
    The elements of $\tilde{\mathcal{P}}_1$ are $\{v_{1, 1}\}$ and $\{v_{2, 1}\}$.  Therefore $B_1(x) = 2x$.
    The elements of $\tilde{\mathcal{P}}_2$ are $\{v_{1, 1}, v_{1, 2}\}$ and $ \{v_{2, 1}, v_{2, 2}\}$.  Therefore $B_2(x) = 2x^2$.

    We will now show that $B_n(x) = x(\sum_{C \in \tilde{\mathcal{P}}_{n-1}}x^{|C|} + 2\sum_{C \in \tilde{\mathcal{P}}_{n-2}}x^{|C|})$ by showing that each element of $\tilde{\mathcal{P}}_n$ can be expressed as the union of a one-element subset of $\{v_{1, n}, v_{2, n}\}$ with an element of $\tilde{\mathcal{P}}_{n-1}$ or $\tilde{\mathcal{P}}_{n-2}$.

    As in the proof of Lemma \ref{A formula}, let $C \in \mathcal{P}_n$, $\hat{C} = C \cap \{v_{1, n}, v_{2, n}\}$, and $C^* = C-\hat{C}$.  Informally speaking, $\hat{C}$ is the elements in $C$ that appear in the last column of $P_{2, n}$ and $C^*$ is the set of elements in $C$ that appear in the first $n-1$ columns of $P_{2, n}$.  By Lemma \ref{2xn flooding conditions}, either $\hat{C} = \{v_{1, n}\}$ or $\hat{C} = \{v_{2, n}\}$.

    If $C^* \cap \{v_{1, n-1}, v_{2, n-1}\} \neq \varnothing$, then $C^* \in \mathcal{P}_{n-1}$.  Exactly one of the two possibilities for $\hat{C}$ will meet the requirements listed in \ref{2xn flooding conditions}.

    If $C^* \cap \{v_{1, n-1}, v_{2, n-1}\} = \varnothing$, then by the parallel path property, $C^* \cap \{v_{1, n-2}, v_{2, n-2}\} \neq \varnothing$ and hence $C^* \in \mathcal{P}_{n-2}$.  Either of the two possibilities for $\hat{C}$ will meet the requirements listed in \ref{2xn flooding conditions}.  Therefore $B_n(x) = x(\sum_{C \in \tilde{\mathcal{P}}_{n-1}}x^{|C|} + 2\sum_{C \in \tilde{\mathcal{P}}_{n-2}}x^{|C|})$ as desired.
\end{proof}

We now have the results necessary to give a recursion for the flood polynomials of parallel path graphs of size $2\times n$.

\begin{thm}
    Using the definition of $A_n(x)$ and $B_n(x)$ given in Lemma \ref{A formula} and Lemma \ref{B formula}, we have $F_{P_{2, n}}(x) = A_n(x)- B_n(x)$.
\end{thm}

\begin{proof}
    By Lemma \ref{2xn flooding conditions}, we know that every flooding cascade set satisfies the parallel path property.  Combining that fact with Lemma \ref{A formula} and Lemma \ref{B formula} gives the desired result.
\end{proof}

We should note that the sequence $F_{P_{2, n}}(1)$ is not currently on the On-Line Encyclopedia of Integers Sequences, but $A_n(1)$ gives sequence A106435 (\cite{OEIS_A106435}) and $\frac{1}{2} B_n(1)$ gives the sequence A001045 (\cite{OEIS_A001045}) which is the Jacobsthal sequence.



We conclude our discussion of parallel path graphs by noting a recursion for $F_{P_{m, n}}(x)$ for $m > 2$ is unknown. The following example demonstrates why this question is hard, even when $m=3$.

\begin{example}
    The following cascade set satisfies the parallel path property as well as conditions similar to those stated in Lemma \ref{2xn flooding conditions}, but it does not flood the graph.  

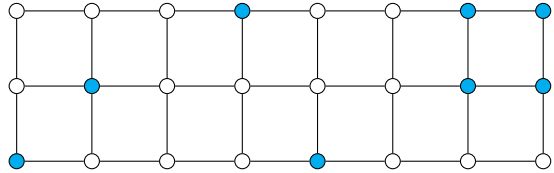
\begin{center}
    
\begin{tikzpicture}
    \node (v00) at (0,0) [flood]  {}; 
    \node (v01) at (0,1) [vertex] {};
    \node (v02) at (0,2) [vertex] {};
    \node (v10) at (1,0) [vertex]  {}; 
    \node (v11) at (1,1) [flood] {};
    \node (v12) at (1,2) [vertex] {};
    \node (v20) at (2,0) [vertex]  {}; 
    \node (v21) at (2,1) [vertex] {};
    \node (v22) at (2,2) [vertex] {};
    \node (v30) at (3,0) [vertex]  {}; 
    \node (v31) at (3,1) [vertex] {};
    \node (v32) at (3,2) [flood] {};
    \node (v40) at (4,0) [flood]  {}; 
    \node (v41) at (4,1) [vertex] {};
    \node (v42) at (4,2) [vertex] {};
    \node (v50) at (5,0) [vertex]  {}; 
    \node (v51) at (5,1) [vertex] {};
    \node (v52) at (5,2) [vertex] {};
    \node (v60) at (6,0) [vertex]  {}; 
    \node (v61) at (6,1) [flood] {};
    \node (v62) at (6,2) [flood] {};
    \node (v70) at (7,0) [vertex]  {}; 
    \node (v71) at (7,1) [flood] {};
    \node (v72) at (7,2) [flood] {};

    \draw [-] (v00) to (v01);
    \draw [-] (v01) to (v02);
    \draw [-] (v10) to (v11);
    \draw [-] (v11) to (v12);
    \draw [-] (v20) to (v21);
    \draw [-] (v21) to (v22);
    \draw [-] (v30) to (v31);
    \draw [-] (v31) to (v32);
    \draw [-] (v40) to (v41);
    \draw [-] (v41) to (v42);
    \draw [-] (v50) to (v51);
    \draw [-] (v51) to (v52);
    \draw [-] (v60) to (v61);
    \draw [-] (v61) to (v62);
    \draw [-] (v70) to (v71);
    \draw [-] (v71) to (v72);

    \draw [-] (v00) to (v10);
    \draw [-] (v10) to (v20);
    \draw [-] (v20) to (v30);
    \draw [-] (v30) to (v40);
    \draw [-] (v40) to (v50);
    \draw [-] (v50) to (v60);
    \draw [-] (v60) to (v70);
    \draw [-] (v01) to (v11);
    \draw [-] (v11) to (v21);
    \draw [-] (v21) to (v31);
    \draw [-] (v31) to (v41);
    \draw [-] (v41) to (v51);
    \draw [-] (v51) to (v61);
    \draw [-] (v61) to (v71);
    \draw [-] (v02) to (v12);
    \draw [-] (v12) to (v22);
    \draw [-] (v22) to (v32);
    \draw [-] (v32) to (v42);
    \draw [-] (v42) to (v52);
    \draw [-] (v52) to (v62);
    \draw [-] (v62) to (v72);

\end{tikzpicture}

\end{center}

\end{example}

\subsection{Cycle} \label{cycle}


A \emph{Cycle Graph} with $n \geq 3 $ vertices, denoted $O_{n}$ 
is the graph with vertex set 
\[V(O_n) = \{v_1, \dots, v_n\}\] 
and edge set 
\[E(O_n) = \{v_1v_{2}, \dots, v_{n-1}v_n, v_nv_1\}.\] 
It follows immediately from construction that $O_n$ is connected and each vertex has degree $2$.

Before providing a recursion for the flood polynomial of $O_n$, we begin with some results about the properties of the flooding cascade sets of $O_n$.

\begin{prop}
    Every neighboring pair of vertices in $O_n$ forms a trigger.
\end{prop}
\begin{proof}
    Suppose we have a neighboring pair of vertices $\{v_x, v_y$\}. Since each vertex has degree $2$, it follows from Definition \ref{deftrigger} that $v_x$ and $v_y$ form a trigger as desired. 
    \end{proof}

%




\begin{lemma} \label{cycleflood}
    Let $C$ be a cascade set of $G = O_n$.  Then $C \in \FG$ if and only if $V-C$ contains no triggers.
\end{lemma}
\begin{proof}
Suppose $C$ is a cascade set of $O_n$ and suppose that $V- C$
contains no triggers.  We want to show that $C \in \mathcal{F}(O_n)$.
We will show this by showing that $v \in C_1$ for all $v \in V$.

If $v \in C$, then this case is trivial as $v \in C$ implies $v \in C_1$.

If $v \notin C$, then because $V-C$ contains no triggers, both of $v$'s neighbors are in $C$. Therefore $v \in C_1$ since it has two neighbors in $C$. 


Therefore, if $V-C$ contains no triggers, then $C \in \mathcal{F}(O_n)$.

If $V-C$ contains a trigger, then by Corollary \ref{Jackson Cor}, $C$ will not flood $O_n$.
\end{proof}

We say that any cascade set of $O_n$ that satisfies the conclusions of Lemma \ref{cycleflood} has the \emph{cycle flood property}.  We now give a recursion for the flood polynomial of cycle graphs.





\begin{thm} \label{Cycle polynomial}
The flood polynomial for cycle graphs can be determine recursively with $F_{O_3}(x) = x^3+3x^2$, $F_{O_4}(x) = x^4+4x^3+2x^2$, and for $n \geq 5$, $F_{O_{n}}(x)=x\cdot F_{O_{n-1}}(x)+x\cdot F_{O_{n-2}}(x)$.
\end{thm}
\begin{proof}
The elements of $\mathcal{F}(O_3)$ are $\{v_{1}, v_2\}, \{v_{2}, v_3\}$, $\{v_{1}, v_{3}\}$ and $\{v_{1}, v_2,  v_{3}\}$.  Therefore $F_{O_3}(x) = x^3+3x^2$.
The elements of $\mathcal{F}(O_4)$ are the cascade sets shown in Example \ref{Flood Set}.  Therefore $F_{O_4}(x) = x^4+4x^3+2x^2$.
    
 Now let $n \geq 5.$  We will show that  $F_{O_n}(x) = x\cdot F_{O_{n-1}}(x)+ x \cdot F_{P_{n-2}}(x)$ by showing that each element of $\mathcal{F}(O_n)$ can be expressed as the union of a single element with either an element of $\mathcal{F}(O_{n-1})$ or $\mathcal{F}(O_{n-2})$.  We will create a map between flooding cascade sets of $O_n$ to those of $O_{n-1}$ or $O_{n-2}$. Recall that from Lemma \ref{cycleflood}, if $C \in \mathcal{F}(O_n)$, then it must contain at least one vertex of every neighboring pair in $O_n$. Applying the map will remove an element from C to create a flooding cascade set of either $O_{n-1}$ or $O_{n-2}$. It will be clear which element was removed from $C$ so the map can be reversed. Note that removing an element in $C$ will decrease $|C|$ by $1$, thus accounting for the factor of $x$ in the equation.  We will keep these facts in mind as we proceed with the proof. We begin by considering $v_n$, with three possible cases for the status of $v_n$.

 Case 1: $v_n  \in C$ and $\{v_{n-1}, v_1\} \subseteq V-C$.  We can map this to the flooding cascade set $C'$ of $O_{n-2}$ where $C' = C - \{v_{n-1}, v_n\}$. Note that since $C$ floods $O_n$, $C$ has the cycle flood property.  This means that $C$ contains at least one element of every neighboring pair of vertices of $O_n$.  Since only one element of $C$ is removed to form $C'$ we only need to verify that at least one of $v_{n-2}$ and $v_1$ is in $C'$ in order for $C'$ to have the cycle flood property.
Since $v_{n-1} \notin C$, it follows that $v_{n-2} \in C$ and hence $v_{n-2} \in C'$ as well. Therefore $C'$ has the cycle flood property and $C' \in \mathcal{F}(O_{n-2})$.  This case accounts for all elements of $\mathcal{F}(O_{n-2})$ that do not contain $v_1$.

 Case 2: $v_n \in C$ and $\{v_{n-1}, v_1\} \nsubseteq V-C$. Then we will map $C$ to the flooding cascade set $C'$ of $O_{n-1}$ where $C' = C -\{v_n\}$. As with the first case, $C'$ has only one fewer element than $C$.  Note again that by Lemma \ref{cycleflood}, $C$ contains at least one element of every neighboring pair of vertices of $O_n$.  Since only one element of $C$ is removed to form $C'$ we only need to verify that at least one of $v_{n-1}$ and $v_1$ is in $C'$ in order for $C'$ to have the cycle flood property.  We are assuming that $\{v_{n-1}, v_1\} \nsubseteq V-C$ so $\{v_{n-1}, v_1\} \nsubseteq V-C'$.  Therefore $C'$ has the cycle flood property and $C' \in \mathcal{F}(O_{n-1})$.  This case accounts for all elements of $\mathcal{F}(O_{n-1})$.

 Case 3: $v_n \notin C$. From Lemma \ref{cycleflood}, we have $\{v_{n-1}, v_1\}\subseteq C$. Then we will map $C$ to the flooding cascade set $C'$ of $O_{n-2}$ where $C' = C -\{v_{n-1}\}$.  As with the first two cases, $C'$ has only one fewer element than $C$.  Note again that by Lemma \ref{cycleflood}, $C$ contains at least one element of every neighboring pair of vertices of $O_n$.  Since only one element of $C$ is removed to form $C'$ we only need to verify that at least one of $v_{n-2}$ and $v_1$ is in $C'$ in order for $C'$ to have the cycle flood property.  We are assuming that $v_1 \in C$ so $v_1 \in C'$.  Therefore $C'$ has the cycle flood property and $C' \in \mathcal{F}(O_{n-2})$.  This case accounts for all elements of $\mathcal{F}(O_{n-2})$ that contain $v_1$.
 
Therefore, every flooding cascade set $C \in O_n$ can be mapped to one of either $O_{n-1}$ or $O_{n-2}$ by removing one element from $C$. This gives us the recursion relationship $F_{O_{n}}(x)=x\cdot F_{O_{n-1}}(x)+x\cdot F_{O_{n-2}}(x)$ as desired. 
 \end{proof}


The \emph{Lucas numbers} are a sequence of number defined recursively by $L_0 = 2$, $L_1 = 1$, and for $n \geq 2$, $L_n = L_{n-1}+L_{n-2}$; the first few Lucas numbers are $2, 1, 3, 4, 7, 11, \dots$.  The Lucas numbers count numerous combinatorial objects \cite{OEIS_A000032}, but in particular, for $n\geq 3$, $L_n$ is the number of independent vertex sets (a set of vertices in a graph where no two of which are adjacent) for the cycle graph $O_n$ \cite{Prodinger_Tichy}.  Note that Lemma \ref{cycleflood} can be rewritten as $C \in \mathcal{F}(O_n)$ if and only if $V(O_n)-C$ is an independent vertex set.  Therefore  $\mathcal{F}(O_n)$ is in bijection with the set of independent vertex sets of $O_n$, leading to the following result enumerating the flood set of $O_n$.

 \begin{cor}
     For $n\geq 3$, $|\mathcal{F}(O_n)| = F_{O_n}(1) = L_n$. 
 \end{cor}

There are many ways to define the \emph{Lucas polynomials}, but one such definition is they are the polynomials defined recursively as $L_0(x) = 2$, $L_1(x) = x$, and for $n \geq 2$, $L_n(x) = x\cdot L_{n-1}(x)+  x\cdot L_{n-2}(x)$. It follows immediately from Theorem \ref{Cycle polynomial} that for $n \geq 3$, $F_{O_n}(x) = L_n(x)$.

We now give a combinatorial interpretation of the coefficients of $L_n(x)$ in terms of flooding cascade sets of $O_n$.


\begin{cor} \label{fibonacci polynomial}
    If $n\geq 3$ and $L_n(x) = \sum_{k=0}^nL(n, k)x^k$, then $L(n, k)$ is the number $k$-element flooding cascade sets of $O_n$. 
\end{cor}

\subsection{Triangle Mosaic} \label{triangle graph}

A \emph{triangle mosaic graph of size $n$} is a graph with $n$  vertices, 
denoted $T_n$. The edge set of $T_n$ is the graph with vertex set 
\[V(T_n) = \{v_1, \dots, v_n\}\] and edge set \[E(T_n) = \{v_iv_j \mid |i-j| \leq 2\}.\]
The following is the triangle mosaic graph of size $6$, $T_6$.
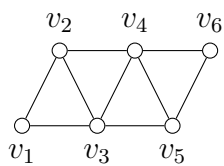
\begin{center}
\begin{tikzpicture}
    \node (v1) at (-4,0) [vertex, label=above: $v_2$]  {}; 
    \node (v2) at (-4.5,-1) [vertex, label=below: $v_1$] {};
    \node (v3) at (-3,0) [vertex, label=above: $v_4$] {};
    \node (v4) at (-3.5,-1) [vertex, label=below: $v_3$] {};
    \node (v5) at (-2,0) [vertex, label=above: $v_6$] {};
    \node (v6) at (-2.5,-1) [vertex, label=below: $v_5$] {};

    \draw [-] (v1) to (v3);
    \draw [-] (v1) to (v2); 
    \draw [-] (v2) to (v4);
    \draw [-] (v3) to (v5);
    \draw [-] (v3) to (v4);
    \draw [-] (v4) to (v6);
    \draw [-] (v5) to (v6);
    \draw [-] (v1) to (v4);
    \draw [-] (v3) to (v6);

\end{tikzpicture}
\end{center}

We saw in Example \ref{diameter} that all triangle mosaic graphs can be flooded with a two-element cascade set.


\begin{lemma} \label{triangle mosaic flood condition}
Let $C$ be a cascade set of $G = T_n$.  If $\{v_k, v_{k+1}\} \subseteq C$ for some $1 \leq k < n$, then $C \in \FG$.
\end{lemma}
\begin{proof}
We will prove this by inducting on $n$. If $n=2$, then $\{v_1, v_2\} = V$ so $\{v_1, v_2\} \in \FG$.
Now suppose that $n>2$ and the result holds for all $T_m$ where $2\leq m < n$.  Let $C$ be a cascade set satisfying $\{v_k, v_{k+1}\} \subseteq C$ for some $1 \leq k < n$.  

If $k = n-1$, then $v_{n-2} \in C_1$ since $v_{n-2}$ is neighbors with both $v_{n-1}$ and $v_n$.  The subgraph $T_n - \{v_n\} \cong T_{n-1}$ and by induction $C_1 - \{v_n\}$ floods this subgraph.  Therefore $C \in \FG$.

A similar argument can be made if $k=1$.

If $1 < k < n-1$, then consider the subgraph $T_n$ consisting of vertices $\{v_1, \dots, v_{k+1}\}$. This subgraph is isomorphic to $T_{k+1}$.  Since $k < n-1$, by induction $C \cap \{v_1, \dots, v_{k+1}\}$ floods this graph.  Similarly, consider the subgraph $T_n$ consisting of vertices $\{v_k, \dots, v_{n}\}$. This subgraph is isomorphic to $T_{n-k+1}$.  Since $k>1$, by induction $C \cap \{v_k, \dots, v_{n}\}$ floods this graph.  Therefore $C \in \FG$ as desired.
\end{proof}


\begin{thm} \label{distance of 4 condition}
Let $C$ be a cascade set of $G = T_n$.  Then $C \in \FG$ if and only if there exists $1 \leq i < j \leq n$ such that $\{v_i, v_j\} \subseteq C$ and $|i-j| \leq 4$.
\end{thm}
\begin{proof}
    Suppose that there exists $1 \leq i < j \leq n$ such that $\{v_i, v_j\} \subseteq C$ and $|i-j| \leq 4$.  If $i = n-1$, then $j=n$ and $C \in \FG$ by Lemma \ref{triangle mosaic flood condition}.
    
    If $i < n-1$, then recall that the neighbors of $v_{i+2}$ are $\{v_i, v_{i+1}, v_{i+3}, v_{i+4}\}$. It follows that $v_{i+2} \in C_1$.  That is because it is either the case that $j = i+2$ or $j \in \{i+1, i+3, i+4\}$.  Since $v_{i+2} \in C_1$, we have that $v_{i+1} \in C_2$.  By Lemma \ref{triangle mosaic flood condition} we have that $C \in \FG$ as desired.
    
    Now suppose that for all $\{v_i, v_j\} \subseteq  C$, it follows that $|i-j| > 4$.  This means that there is no element of $V-C$ that is neighbors with at least two elements of $C$.  Therefore $C_1 = C$ and $C \notin \FG$.
\end{proof}

We can now give a formula for the flood polynomial of $T_n$, but for convenience, let $\textrm{COMP}(n, 4)$ be defined as follows: 
\[\textrm{COMP}(n, 4) = \{\alpha \vDash (n+1) \mid  \alpha_k \leq 4 \text{ for some } 1<k<\ell(\alpha) \}.\]
Note that the proof of the following theorem will rely on the notation introduced in Section \ref{comps and partitions}.

\begin{thm}
    The flood polynomial for $T_n$ is given by 
    \[F_{T_n}(x) = \sum_{\alpha \in \textrm{COMP}(n, 4)} x^{\ell(\alpha)-1}.\]
\end{thm}

\begin{proof}
    We saw in Theorem \ref{distance of 4 condition} that $C \in \mathcal{F}(T_n)$ if and only if there exists $1 \leq i < j \leq n$ such that $\{v_i, v_j\} \subseteq  C$ and $|i-j| \leq 4$.  Let $S(C)$ be the set of indices of the elements of $C$.  For example, if $C = \{v_1, v_4\}$, then $S(C) = \{1, 4\}$.  It follows from the definitions of $S(C)$ and $\textrm{COMP}(n, 4)$ that $S(C) \subseteq [n]$ and $C \in \mathcal{F}(T_n)$ if and only if $\co({S(C)}) \in \textrm{COMP}(n, 4)$.  To conclude the proof, note that $\ell(\co({S(C)})) -1  = |C|$.
\end{proof}


\section{Graphs with the Same Flood Polynomial} \label{Graphs with the Same Flood Polynomial}

In this section we provide general examples of pairs of distinct graphs with the same flood polynomial.  Due to the large number of graphs and limits to the sizes of the coefficients (see Proposition \ref{coefficients}), it is common for a graph to share a flood polynomial with another graph.  The smallest example is shown in Example \ref{2 element graphs}.

Before we get to the main results of this section, we need a lemma that will be used in Section \ref{centipede} and Section \ref{tick}.  

\begin{lemma} \label{Shaboingboing}
    Suppose $G$ is a graph that contains a vertex $v$ with the following property: $\deg(v)=2m+2$ and exactly $m+2$ of $v$'s neighbors are leaves. Let $\{l_1, \dots l_{m+2}\}$ be the set of leaves that are neighbors with $v$ and let $\{g_1, \dots, g_m\}$ be the set of neighbors of $v$ that are not leaves.  Then $\F = F_{P_3}(x)\cdot F_{G'}(x)$ where $G'$ is formed by removing $v$, $l_{m+1}$, and $l_{m+2}$ and all edges incident to $v$ from $G$ and then adding edges between $l_k$ and $g_k$ for $1 \leq k \leq m$.
\end{lemma}

The example below illustrates the formation of $G'$ as described in Lemma \ref{Shaboingboing}.

\begin{example}
    Let $G$ be a graph with the property described in Lemma \ref{Shaboingboing}.  We highlighted in orange the three vertices that are removed from $G$ to create $G'$.  You can see how the new edges are formed in $G'$ when the vertices are removed.

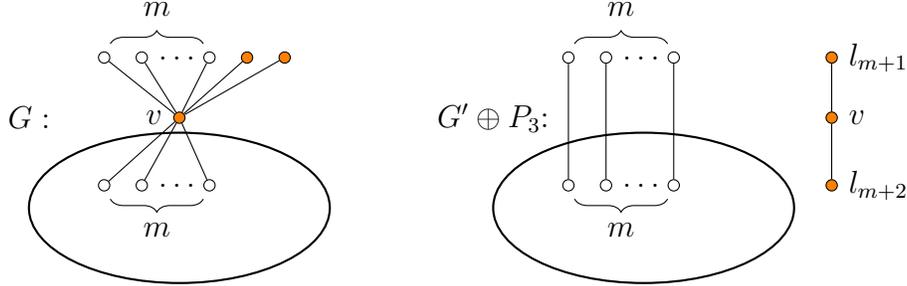
\begin{center}
\begin{tikzpicture} [auto, vertex/.style={circle,draw=black!100,thin,inner sep=0pt,minimum size=1.5mm}, flood/.style={circle,draw=black!100,fill=cyan!100, thin,inner sep=0pt,minimum size=1.5mm}]
    \node[ellipse, thick,
    draw = black,
    minimum width = 4cm, 
    minimum height = 2cm] (e) at (0,0) {};

    \node at (-2, 1.2) {$G:$};

    \node (v) at (0,1.2) [vertex, fill = orange, label=left: $v$]  {}; 

    \node (l1) at (-1,2) [vertex]  {}; 
    \node (l2) at (-.5,2) [vertex]  {}; 
    \node at (0, 2) {$\dots$};
    \node (l3) at (.4,2) [vertex]  {};
    \node (l4) at (.9,2) [vertex, fill = orange]  {};
    \node (l5) at (1.4,2) [vertex, fill = orange]  {};

    \node (g1) at (-1,.3) [vertex]  {}; 
    \node (g2) at (-.5,.3) [vertex]  {}; 
    \node at (0, .3) {$\dots$};
    \node (g3) at (.4,.3) [vertex]  {};

    \draw [-] (v) to (l1);
    \draw [-] (v) to (l2);
    \draw [-] (v) to (l3);
    \draw [-] (v) to (l4);
    \draw [-] (v) to (l5);
    \draw [-] (v) to (g1);
    \draw [-] (v) to (g2);
    \draw [-] (v) to (g3);

    \draw [decorate,decoration={brace,amplitude=5pt,mirror,raise=1ex}]
  (g1) -- (g3) node[midway,yshift=-2em]{$m$};

  \draw [decorate,decoration={brace,amplitude=5pt,raise=1ex}]
  (l1) -- (l3) node[midway,yshift=1em]{$m$};
    
\end{tikzpicture}
\hspace{1 cm}
\begin{tikzpicture} [auto, vertex/.style={circle,draw=black!100,thin,inner sep=0pt,minimum size=1.5mm}, flood/.style={circle,draw=black!100,fill=cyan!100, thin,inner sep=0pt,minimum size=1.5mm}]
    \node[ellipse, thick,
    draw = black,
    minimum width = 4cm, 
    minimum height = 2cm] (e) at (0,0) {};

    \node at (-2, 1.2) {$G' \oplus P_3$:};


    \node (l1) at (-1,2) [vertex]  {}; 
    \node (l2) at (-.5,2) [vertex]  {}; 
    \node at (0, 2) {$\dots$};
    \node (l3) at (.4,2) [vertex]  {};
    
    \node (v) at (2.5,1.2) [vertex, fill = orange, label=right: $v$]  {}; 
    \node (l4) at (2.5,2) [vertex, fill = orange, label=right: $l_{m+1}$]  {};
    \node (l5) at (2.5,.3) [vertex, fill = orange, label=right: $l_{m+2}$]  {};

    \node (g1) at (-1,.3) [vertex]  {}; 
    \node (g2) at (-.5,.3) [vertex]  {}; 
    \node at (0, .3) {$\dots$};
    \node (g3) at (.4,.3) [vertex]  {};

    \draw [-] (v) to (l4);
    \draw [-] (v) to (l5);
    
    \draw [-] (g1) to (l1);
    \draw [-] (g2) to (l2);
    \draw [-] (g3) to (l3);

    \draw [decorate,decoration={brace,amplitude=5pt,mirror,raise=1ex}]
  (g1) -- (g3) node[midway,yshift=-2em]{$m$};

  \draw [decorate,decoration={brace,amplitude=5pt,raise=1ex}]
  (l1) -- (l3) node[midway,yshift=1em]{$m$};
    
\end{tikzpicture}
\end{center}
    
\end{example}

\begin{proof}[Proof of Lemma \ref{Shaboingboing}]
Since $G$ and $G' \oplus P_3$ have the same vertex set, we can prove this result by showing that $C \in \FG$ if and only if $C \in \mathcal{F}(G' \oplus P_3)$.

Suppose for contradiction that there exists a flooding cascade set $C \in \FG$ such that $C \notin \mathcal{F}(G' \oplus P_3)$ and we will assume that $C$ is largest such set.  That is to say if $|C'| > |C|$ and $C' \in \FG$, then $C' \in \mathcal{F}(G' \oplus P_3)$.  Clearly $|C| \neq n$ because if $|C| = n$, then $C$ is trivially an element of $\mathcal{F}(G' \oplus P_3)$.  To avoid confusion, we will write $C_1(G)$ for the first element in the flood sequence of $C$ on the graph $G$.  Since $|C|<n$ and $C \in \FG$, there must be an element $x \in C_1(G) - C$.  Note that by Proposition \ref{flooding superset}, $C_1(G) \in \FG$, so if we can conclude that $x \in C_1(G' \oplus P_3)$ (and hence $C_1(G) \subseteq C_1(G' \oplus P_3)$), then we can conclude  that $C \in \mathcal{F}(G' \oplus P_3)$.  This is because $|C_1(G)| > |C|$.


There are three cases to consider: $x = v$, $x \in \{g_1, \dots, g_m\}$, or $x$ is not a neighbor of $v$.  If $x=v$, then $x \in C_1(G'\oplus P_3)$ because $l_{m+1}$ and $l_{m+2}$ are leaves so they must be in $C$.  If $x \in \{g_1, \dots, g_m\}$, the $C$ contains at least one neighbor of $x$ not equal to $v$.  Without loss of generality we can say $x = g_1$.  Since $C$ contains at least one neighbor of $x$ not equal to $v$ and $C$ contains $l_1$, we have that $x \in C_1(G' \oplus P_3).$  The final case to consider is $x$ is not a neighbor of $v$.  In this case, $x \in C_1(G' \oplus P_3)$ since the only new edges created in $G'$ involve $v$.

Therefore if $C \in \FG$, then $C \in \mathcal{F}(G' \oplus P_3)$.

Now suppose for contradiction that there exists a flooding cascade set $C \in \mathcal{F}(G' \oplus P_3)$ such that $C \notin \FG$ and we will assume that $C$ is largest such set.  That is to say if $|C'| > |C|$ and $C' \in \mathcal{F}(G' \oplus P_3)$, then $C' \in \FG$.  Clearly $|C| \neq n$ because if $|C| = n$, then $C$ is trivially an element of $\FG$.

If $v \notin C$, then $v \in C_1(G' \oplus P_3)$ since $l_{m+1}$ and $l_{m+2}$ are leaves so they must be in $C$.  Similarly, $v \in C_1(G)$.  Note that $C \cup \{v\} \in \FG$ since $|C \cup \{v\}| > |C|$ and $C_1(G) \supseteq C \cup\{v\}$. So $C_1(G) \in \FG$ and $C \in \FG$.

Now suppose $v \in C$.  Let $x \in C_1(G' \oplus P_3)-C$.  It's either the case that $x \in \{g_1, \dots, g_m\}$ or $x$ is not a neighbor of $v$.  If $x \in \{g_1, \dots, g_m\}$, then without loss of generality we can say $x = g_1$.  Then $C$ contains at least one neighbor of $x$ that is not equal to $l_1$ and also contains $v$.  Therefore $x \in C_1(G)$.  If $x$ is not a neighbor of $v$, then $x \in C_1(G)$ since the only new edges created in $G'$ involve $v$.  

Therefore if $C \in \mathcal{F}(G' \oplus P_3)$, then $C \in \FG$ as desired. 
\end{proof}

\subsection{Path graphs with an even number of vertices} \label{Path graphs with an even number of vertices}

We now show that the flood polynomial for the path with $2n$ vertices has the same flood polynomial as the disjoint union of the path with $n$ vertices with the cycle with $n$ vertices.  This will lead to an alternate proof of a well-known result about even-indexed Fibonacci polynomials that $f_{2n}(x) = f_n(x)\cdot L_n(x)$ (see \cite{Benjamin Quinn}).

\begin{thm} \label{An Alternate Proof that $F_{2n}(x) = F_n(x)L_n(x)$}
    For $n \geq 3, F_{P_{2n}}(x) = F_{P_n}(x) \cdot F_{O_n}(x)$. 
\end{thm}
\begin{proof}
    Recall that both $\mathcal{F}(P_{2n})$ and $\mathcal{F}(P_n\oplus O_n)$ are sets of flooding cascade sets. We will create a bijection between $\mathcal{F}(P_{2n})$ and $\mathcal{F}(P_n\oplus O_n)$ that preserves the size of the flooding cascade set.  Once the bijection is established, the result follows immediately from Proposition \ref{disconnected}. 

    Let $\{v_1, \dots v_{2n}\}$ be the set of vertices of $P_{2n}$ and let $\{a_1, \dots a_{n}\} \cup \{b_1, \dots b_{n}\}$ be the set vertices of $P_n\oplus O_n$. The edge set of $P_{2n}$ is $\{v_1v_2, \dots v_{2n-1}v_{2n}\}$ and the edge set of $P_n\oplus O_n$ is $\{a_1a_2, \dots a_{n-1}a_{n}\} \cup \{b_{1}b_{2}, \dots b_{n-1}b_{n}, b_{1}b_{n}\}$.  That is to say, the $a$ vertices are the vertices of the path and the $b$ vertices are the vertices of the cycle.

    
    Let us begin by mapping the flooding cascade sets in $\mathcal{F}(P_{2n})$ to those of $F(P_n\oplus O_n)$. Let $C\in\mathcal{F}(P_{2n})$. Since $P_{2n}$ is a path graph, by Theorem \ref{PPP Path} we know that $C$ has the parallel path property so both $v_1$ and $v_{2n}$ are in $C$ and that for all $1<i<2n$, if $v_i \notin C$, then $v_{i-1} \in C$ and $v_{i+1} \in C$.  There are two cases to consider when describing our map: $v_n \in C$, and $v_n \notin C$.



If $v_n \in C$, then let $A$ be the set defined by for $1\leq i \leq n$, $a_i \in A$ if and only if $v_i \in C$.  Similarly, let $B$ be the set defined by for $1\leq i \leq n$, $b_i \in B$ if and only if $v_{n+i} \in C$.  Note that $|C| = |A| + |B|$.  Since $C$ has the parallel path property and $\{a_1, a_n\} \subseteq A$, it follows that $A$ has the parallel path property so $A \in \mathcal{F}(P_n)$.  Since $v_{2n} \in C$, it follows that $b_n \in B$ and $B$ has the cycle flood property.  Therefore $B \in \mathcal{F}(O_n)$.  This case accounts for all elements of $B \in \mathcal{F}(O_n)$ where $b_n \in B$.

If $v_n \notin C$, then since $C$ has the parallel path property, it follows that $v_{n+1} \in C$.  Let $A$ be the set defined by for $1\leq i \leq n$, $a_i \in A$ if and only if $v_{n+i} \in C$.  Similarly, let $B$ be the set defined by for $1\leq i \leq n$, $b_i \in B$ if and only if $v_{i} \in C$.  As with the previous case, note that $|C| = |A| + |B|$.  Since $C$ has the parallel path property and $\{a_1, a_n\} \subseteq A$, it follows that $A$ has the parallel path property so $A \in \mathcal{F}(P_n)$.  Since $v_{1} \in C$, it follows that $b_1 \in B$ and $B$ has the cycle flood property.  Therefore $B \in \mathcal{F}(O_n)$.  This case accounts for all elements of $B \in \mathcal{F}(O_n)$ where $b_n \notin B$.

In both of these cases, the map can easily be reversed.  Therefore there exists a size-preserving bijection between the flooding cascade sets of $P_{2n}$ and those of $P_n$ and $O_n$, so $F_{P_{2n}}(x) = F_{P_n}(x) \cdot F_{O_n}(x)$. 
\end{proof}

Subbing in $x=1$ gives the following relation between Fibonacci and Lucas numbers.

\begin{cor}
    For all $n \geq 3$, $f_{2n}= f_n\cdot L_n$.
\end{cor}

In Question \ref{Other Fib Polys}, we give some ideas for how Theorem \ref{An Alternate Proof that $F_{2n}(x) = F_n(x)L_n(x)$} may be generalized.

\subsection{Centipede} \label{centipede}

A \emph{centipede graph of type $\alpha \vDash (n-1)$} for $n \geq 3$ is a graph with \\ $n+4\cdot(\ell(\alpha)-1)$ vertices, 
denoted $\cat_{\alpha}$. It is the graph with vertex set \[V(\cat_\alpha) = \{v_1, \dots, v_n\} \cup \{l_{d,1}, \dots l_{d,4} \mid d \in D(\alpha)\}\] and edge set 
\[E(\cat_\alpha)=\{v_1v_2, \dots, v_{n-1}v_n\} \cup \{v_{d+1}l_{d,1}, \dots v_{d+1}l_{d,4} \mid d \in D(\alpha)\}.\]
In plain words, $\cat_\alpha$ is the graph that can be constructed by starting with an $n$-element path graph and then appending four leaves to each $v_{d+1}$ where $d \in D(\alpha)$.

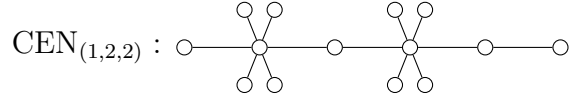
\begin{center}
\begin{tikzpicture}
\node at (-1.3, 0) {$\cat_{(1, 2 , 2)}:$};
\node (v1) at ( 0,0) [vertex] {};
\node (v2) at ( 1,0) [vertex] {};

\node (l21) at ( .8,.5) [vertex] {};
\node (l22) at ( 1.2,.5) [vertex] {};
\node (l23) at ( .8,-.5) [vertex] {};
\node (l24) at ( 1.2,-.5) [vertex] {};

\node (v3) at ( 2,0) [vertex] {};
\node (v4) at ( 3,0) [vertex] {};

\node (l41) at ( 2.8,.5) [vertex] {};
\node (l42) at ( 3.2,.5) [vertex] {};
\node (l43) at ( 2.8,-.5) [vertex] {};
\node (l44) at ( 3.2,-.5) [vertex] {};

\node (v5) at ( 4,0) [vertex] {};
\node (v6) at ( 5,0) [vertex] {};
\draw [-] (v1) to (v2);
\draw [-] (v2) to (v3);
\draw [-] (v3) to (v4);
\draw [-] (v4) to (v5);
\draw [-] (v5) to (v6);

\draw [-] (v2) to (l21);
\draw [-] (v2) to (l22);
\draw [-] (v2) to (l23);
\draw [-] (v2) to (l24);

\draw [-] (v4) to (l41);
\draw [-] (v4) to (l42);
\draw [-] (v4) to (l43);
\draw [-] (v4) to (l44);

\end{tikzpicture}
\end{center}

\begin{thm} \label{Cent Polynomial}
    The centipede graph $\cat_\alpha$ has the same flood polynomial as the disjoint union of $2\cdot\ell(\alpha) - 1$ path graphs.  In particular,
    \[ F_{\cat{\alpha}}(x) = (F_{P_3}(x))^{\ell(\alpha)-1}\cdot\prod_{j=1}^{l(\alpha)}F_{P_{\alpha_j+1}}(x) \]
\end{thm}

\begin{proof}
We will prove this by inducting on $\ell(\alpha)$.

If $\ell(\alpha) = 1$, then $\alpha = (n-1)$ and $\cat_\alpha \cong P_{n}$.  Therefore $F_{\cat_{(n-1)}}(x) = F_{P_n}(x)$ as desired.

Now suppose $\ell(\alpha) > 1$ and \[F_{\cat{\beta}}(x) = (F_{P_3}(x))^{\ell(\beta)-1}\cdot\prod_{j=1}^{l(\beta)}F_{P_{\beta_j+1}}(x)\] whenever $\ell(\beta) < \ell(\alpha)$.

Let $\alpha'$ be the composition defined by $\alpha' =(\alpha_2, \cdots, \alpha_{\ell(\alpha)})$.
Note that applying Lemma \ref{Shaboingboing} with $v=v_{\alpha_1+1}$ gives us that \[F_{\cat{\alpha}}(x) = F_{P_3}(x)\cdot F_{P_{\alpha_1+1}}(x) \cdot F_{\cat{\alpha'}}(x).\]

Since $\ell(\alpha')=\ell(\alpha)-1<\ell(\alpha)$, we have that 
\[ F_{\cat{\alpha'}}(x) = (F_{P_3}(x))^{\ell(\alpha')-1}\cdot\prod_{j=1}^{l(\alpha')}F_{P_{\alpha'_j+1}}(x) = (F_{P_3}(x))^{\ell(\alpha)-2}\cdot\prod_{j=2}^{l(\alpha)}F_{P_{\alpha_j+1}}(x). \]

Hence \[ F_{\cat{\alpha}}(x) = (F_{P_3}(x))^{\ell(\alpha)-1}\cdot\prod_{j=1}^{l(\alpha)}F_{P_{\alpha_j+1}}(x) \] as desired.
\end{proof}

Notice that the values of the entries of $\alpha$ have an effect on the flood polynomial of $\cat_{\alpha}$, but the order in which they appear does not.

\begin{cor}
    If $\alpha \sim \lambda$ and $\beta \sim \lambda$, then $F_{\cat_{\alpha}}(x) = F_{\cat_{\beta}}(x)$.
\end{cor}

Combining Theorem \ref{Cent Polynomial} with the results of Section \ref{Pn} we get that the flood polynomial of centipede graphs is the product of Fibonacci polynomials and, as a result, the size of the flood set is a product of Fibonacci numbers.

\begin{cor}
    The number of flooding cascade sets of $\cat_{\alpha}$ is a product of Fibonacci numbers.  In particular 
    \[|\mathcal{F}(\cat_\alpha)|= f_3^{\ell(\alpha)-1}\cdot\prod_{j=1}^{l(\alpha)}f_{\alpha_j+1} .\]
\end{cor}


\subsection{Tick Graph} \label{tick}

A \emph{tick graph of size $\alpha \vDash n$} for $n \geq 3$ is a graph with $n+4\cdot(\ell(\alpha))$ vertices, 
denoted $\tick_{\alpha}$. Let $D = (D(\alpha)\cup \{n\})$. The tick graph $\tick_\alpha$ is the graph with vertex set 
\[V(\tick_\alpha) = \{v_1, \dots, v_n\} \cup \{l_{d,1}, \dots, l_{d,4} \mid d \in D\}\]
and edge set 
\[E(\tick_\alpha)= \{v_1v_2, \dots, v_{n-1}v_n, v_nv_1\} \cup \{v_dl_{d,1}, \dots, v_dl_{d,4} \mid d \in D\}.\]
In plain words, $\tick_\alpha$ can be constructed by starting with an $n$-element cycle graph and then appending four leaves to $v_d$ for all $d \in D$.

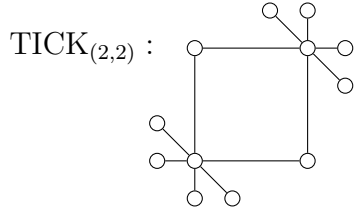
\begin{center}
\begin{tikzpicture}
\node at (-1.5, 1.5) {$\tick_{(2, 2)}:$};

\node (v1) at ( 0,0) [vertex] {};
\node (v2) at ( 1.5,0) [vertex] {};
\node (v3) at ( 0,1.5) [vertex] {};
\node (v4) at ( 1.5,1.5) [vertex] {};

\node (l11) at ( -.5,.5) [vertex] {};
\node (l12) at ( -.5,0) [vertex] {};
\node (l13) at ( 0,-.5) [vertex] {};
\node (l14) at ( .5,-.5) [vertex] {};

\node (l41) at ( 1,2) [vertex] {};
\node (l42) at ( 1.5,2) [vertex] {};
\node (l43) at ( 2, 1) [vertex] {};
\node (l44) at ( 2,1.5) [vertex] {};

\draw [-] (v1) to (v2);
\draw [-] (v1) to (v3);
\draw [-] (v3) to (v4);
\draw [-] (v2) to (v4);

\draw [-] (v1) to (l11);
\draw [-] (v1) to (l12);
\draw [-] (v1) to (l13);
\draw [-] (v1) to (l14);

\draw [-] (v4) to (l41);
\draw [-] (v4) to (l42);
\draw [-] (v4) to (l43);
\draw [-] (v4) to (l44);

\end{tikzpicture}
\end{center}

\begin{thm} \label{Tick Polynomial}
    The tick graph $\tick_\alpha$ has the same flood polynomial as the disjoint union of $2\cdot\ell(\alpha)$ path graphs.  In particular,
    \[ F_{\tick{\alpha}}(x) = (F_{P_3}(x))^{\ell(\alpha)}\cdot\prod_{j=1}^{l(\alpha)}F_{P_{\alpha_j+1}}(x) \]
\end{thm}

\begin{proof}
Applying Lemma \ref{Shaboingboing} with $v = v_n$ gives
\[F_{\tick{\alpha}}(x) = F_{P_3} \cdot F_{\cat{\alpha}}(x).\]
It follows from Theorem \ref{Cent Polynomial} that 
\[ F_{\tick{\alpha}}(x) = (F_{P_3}(x))^{\ell(\alpha)}\cdot\prod_{j=1}^{l(\alpha)}F_{P_{\alpha_j+1}}(x). \]
\end{proof}

As with the case of the centipede graph, the values of the entries of $\alpha$ have an effect on the flood polynomial of $\tick_{\alpha}$, but the order in which they appear does not.

\begin{cor}
    If $\alpha \sim \lambda$ and $\beta \sim \lambda$, then $F_{\tick_{\alpha}}(x) = F_{\tick_{\beta}}(x)$.
\end{cor}

Combining Theorem \ref{Tick Polynomial} with the results of Section \ref{Pn} we get that the flood polynomial of tick graphs is the product of Fibonacci polynomials and as a result, the size of the flood set is a product of Fibonacci numbers.

\begin{cor}
    The number of flooding cascade sets of $\tick_{\alpha}$ is a product of Fibonacci numbers.  In particular 
    \[|\mathcal{F}(\tick_\alpha)|= f_3^{\ell(\alpha)}\cdot\prod_{j=1}^{l(\alpha)}f_{\alpha_j+1} .\]
\end{cor}

It follows immediately from the proof of Theorem \ref{Tick Polynomial}, that every tick graph has the same flood polynomial as the disjoint union of a three-element path graph with a centipede graph.

\begin{cor}
    If $\alpha \sim \lambda$ and $\beta \sim \lambda$, then $F_{\tick_{\alpha}}(x) = F_{P_3}(x)\cdot F_{\cat{\beta}}(x)$.
\end{cor}




\section{Discussion and Open Questions}

We conclude this article with some open questions, the difficulty of which remain unclear.

\begin{quest}
    Can the results of Section \ref{leaves and triggers} be generalized to determine the number of certain three-element subgraphs of $G$ from $\F$?
\end{quest}

\begin{quest} \label{Other Fib Polys}
    Since $F_{P_n}(x)$ is a Fibonacci polynomial, it is well-known that if $m$ divides $n$, then $F_{P_m}(x)$ divides $F_{P_n}(x)$ \cite{Bicknell}. For what $m$ and $n$ does there exist a graph $H$ where $F_{P_n}(x)$ = $F_{P_m}(x)\cdot F_H(x)$? The case where $n=2m$ is Theorem \ref{An Alternate Proof that $F_{2n}(x) = F_n(x)L_n(x)$}.  When $m=3$ and $n=9$, there is no such graph.
\end{quest}

\begin{quest}
    For what polynomials $p(x)$ does there exist a graph $G$ such that $\F = p(x)$.  We know from Proposition \ref{coefficients} some necessary conditions that the coefficients of $p(x)$ must satisfy, but can these polynomials be completely classified?
\end{quest}

In \cite{Alexander Hearding}, the authors introduce chainsaw graphs and broken chainsaw graphs.  They proved that the number of independent vertex sets of these graphs is enumerated by generalized Fibonacci and Lucas numbers.  We saw in Section \ref{Pn} and Section \ref{cycle} that the independent vertex sets of path graphs and cycle graphs are closely related to the flooding cascade sets.  Chainsaw and broken chainsaw graphs are natural generalizations of cycle graphs and path graphs, respectively.
\begin{quest}
    Is there a recursive formula for the flood polynomial of chainsaw graphs or broken chainsaw graphs that generalize the results of Section \ref{Pn} and Section \ref{cycle}?
\end{quest}

\section{Acknowledgments}
The authors would like to thank Molly Lynch for providing feedback during the editing process.


\begin{thebibliography}{3}

\bibitem{Alexander Hearding} 
Alexander, J. and Hearding, P.  \textit{A Graph-Theoretic Encoding of Lucas Sequences}, The Fibonacci Quarterly, vol. 53, no. 3, 2015, pp. 237-240. 

\bibitem{Benjamin Quinn} 
Benjamin, A. T. and Quinn, J. J. \textit{Proofs that Really Count: The Art of Combinatorial Proof}, Dolciani Mathematical Expositions, Vol. 27, Mathematical Association of America, p. 141, ISBN 978-0-88385-333-7.

\bibitem{Bicknell} 
Bicknell, M.  \textit{A Primer for the Fibonacci Numbers -- Part VII}, The Fibonacci Quarterly, vol. 8, no. 4, 1970, pp. 407–420.

\bibitem{EoM_FibPolynomials} 
Fibonacci polynomials. Encyclopedia of Mathematics. \\ \texttt{http://encyclopediaofmath.org/index.php?title=Fibonacci\_polynomials\&oldid=55342}

\bibitem{OEIS_A011973} 
OEIS Foundation Inc. (2025), Coefficients of (one version of) Fibonacci polynomials, Entry A011973 in The On-Line Encyclopedia of Integer Sequences, \texttt{https://oeis.org/A011973}.


\bibitem{OEIS_A106435}
OEIS Foundation Inc. (2025), Entry A106435 in The On-Line Encyclopedia of Integer Sequences, \texttt{https://oeis.org/A106435}.


\bibitem{OEIS_A000045} 
OEIS Foundation Inc. (2025), Fibonacci numbers, Entry A000045 in The On-Line Encyclopedia of Integer Sequences, \texttt{https://oeis.org/A000045}.



\bibitem{OEIS_A000032} 
OEIS Foundation Inc. (2025), Lucas numbers, Entry A000032 in The On-Line Encyclopedia of Integer Sequences, \texttt{https://oeis.org/A000032}.


\bibitem{OEIS_A001045}
OEIS Foundation Inc. (2025), Jacobsthal sequence, Entry A001045 in The On-Line Encyclopedia of Integer Sequences, \texttt{https://oeis.org/A106435}.



\bibitem{Prodinger_Tichy} 
Prodinger, H. and Tichy, R.  \textit{Fibonacci Numbers of Graphs}, The Fibonacci Quarterly, vol. 20, no. 1, 1982, pp. 16–21.

\bibitem{Stanley}
Stanley, R.P. (2011). {\it Enumerative Combinatorics.} Second edition, vol 1. Cambridge University Press.


























\end{thebibliography}
\end{document}